\newif\ificmlver
\icmlvertrue

\newcommand{\titlename}{Two-timescale Derivative Free Optimization for Performative Prediction with Markovian Data}

\ificmlver
\documentclass{article}

\usepackage{natbib}
\usepackage{fullpage}
\usepackage{algorithmic}
\usepackage[ruled,vlined,linesnumbered,noresetcount]{algorithm2e}

\usepackage[utf8]{inputenc} 
\usepackage[T1]{fontenc}    
\usepackage[colorlinks=true, citecolor = blue]{hyperref}       
\usepackage{url}            
\usepackage{booktabs}       
\usepackage{amsfonts}       
\usepackage{nicefrac}       
\usepackage{microtype}      
\usepackage[table]{xcolor}         
\usepackage{graphicx}
\usepackage{comment}
\usepackage{mathtools}
\usepackage{caption}

\title{\titlename}
\author{Haitong Liu, Qiang Li, Hoi-To Wai \thanks{H.T. ~Liu is from the Department of Computer Science and Engineering, Q.~Li and H.-T.~Wai are with the Department of Systems Engineering and Engineering Management, The Chinese University of Hong Kong, Hong Kong SAR of China. Emails: \texttt{antonyhtliu@link.cuhk.edu.hk, \{liqiang, htwai\}@se.cuhk.edu.hk}}}

\usepackage{amsmath, amssymb, bm, bbm, multicol, Shortcuts_OPT, wrapfig, enumitem}

\else
\documentclass{article}

\usepackage{natbib}
\usepackage{fullpage}

\usepackage[utf8]{inputenc} 
\usepackage[T1]{fontenc}    
\usepackage{hyperref}       
\usepackage{url}            
\usepackage{booktabs}       
\usepackage{amsfonts}       
\usepackage{nicefrac}       
\usepackage{microtype}      
\usepackage{graphicx}
\usepackage{dsfont}
\usepackage{vwcol}
\usepackage{changepage}

\usepackage{enumitem,url,amsmath,amssymb,amsthm}

\usepackage{mdframed}

\usepackage{shortcuts_OPT, enumitem}

\newlength\figH
\newlength\figW
\usepackage[utf8]{inputenc}
\usepackage{pgfplots}
\DeclareUnicodeCharacter{2212}{−}
\usepgfplotslibrary{groupplots,dateplot}
\usetikzlibrary{patterns,shapes.arrows}
\pgfplotsset{compat=newest}

\usepackage{psfrag,subfigure,float,hyperref,bbm,cleveref}
\hypersetup{
  colorlinks   = true, 
  urlcolor     = blue, 
  linkcolor    = blue, 
  citecolor    = blue   
}

\makeatletter
\def\Exa@space@setup{%
  \Exa@preskip=5cm plus 2cm minus 2cm
  \Exa@postskip=\Exa@preskip
}
\makeatother

\title{\titlename} 
\date{\today}

\fi

\allowdisplaybreaks[4]

\usepackage{amsmath,amsthm}

\theoremstyle{plain}
\newtheorem{theorem}{Theorem}[section]

\newtheorem{lemma}{Lemma}[section]
\newtheorem{Corollary}{Corollary}[section]
\theoremstyle{definition}

\newtheorem{assumption}{Assumption}[section]
\theoremstyle{remark}
\newtheorem{remark}{Remark}[section]

\usepackage{microtype}
\usepackage{graphicx}
\usepackage{subfigure}
\usepackage{booktabs}
\usepackage{natbib}

\begin{document}

\maketitle

\begin{abstract}
This paper studies the performative prediction problem where a learner aims to minimize the expected loss with a decision-dependent data distribution. Such setting is motivated when outcomes can be affected by the prediction model, e.g., in strategic classification. We consider a state-dependent setting where the data distribution evolves according to an underlying controlled Markov chain. We focus on stochastic derivative free optimization (DFO) where the learner is given access to a loss function evaluation oracle with the above Markovian data. We propose a two-timescale DFO($\lambda$) algorithm that features {\sf (i)} a sample accumulation mechanism that utilizes every observed sample to estimate the overall gradient of performative risk,
and {\sf (ii)} a two-timescale diminishing step size that balances the rates of DFO updates and bias reduction. Under a general non-convex optimization setting, we show that DFO($\lambda$) requires ${\cal O}( 1 /\epsilon^3)$ samples (up to a log factor)  to attain a near-stationary solution with expected squared gradient norm less than $\epsilon > 0$. Numerical experiments verify our analysis.
\end{abstract}

\section{Introduction}
\label{sec:submission}
Consider the following stochastic optimization problem with decision-dependent data:
\beq\label{perf}
\min_{ \prm \in \RR^d } ~ {\cal L}(\prm) = \EE_{ Z \sim \Pi_{\prm} } \big[ \ell( \prm; Z ) \big].
\eeq
Notice that the decision variable $\prm$ appears in both the loss function $\ell( \prm; Z )$ and the data distribution $\Pi_{\prm}$ (denoted by $D(\prm)$ in some prior literature) supported on ${\sf Z}$. 
The overall loss function ${\cal L}(\prm)$ is known as the \emph{performative risk} which captures the distributional shift due to changes in the deployed model. This setting is motivated by the recent studies on \emph{performative prediction} \citep{perdomo2020performative}, which considers outcomes $Z$ that are supported by the deployed model $\prm$. For example, it covers strategic classification \citep{hardt2016strategic, dong2018strategic} in economics and financial practices such as with the training of loan classifier for customers who may react to the deployed model $\prm$ to maximize their gains; or in price promotion mechanism \citep{Zhang2018PriceProm} where customers react to prices with the aim of gaining a lower price; or in ride sharing business \citep{narang2022multiplayer} with customers who adjust their demand according to prices set by the platform.


Due to the effects of $\prm$ on both the loss function and distribution, the objective function ${\cal L}(\prm)$ is non-convex in general. Numerous efforts have been focused on characterizing and finding the so-called \emph{performative stable} solution which is a fixed point to the repeated risk minimization (RRM) process \citep{perdomo2020performative, mendler2020stochastic, brown2022performative, li2022state, roy2022projection, drusvyatskiy2022stochastic}. While RRM might be a natural algorithm for scenarios when the learner is agnostic to the performative effects in the dynamic data distribution, the obtained solution maybe far from being optimal or stationary to \eqref{perf}. 

On the other hand, recent works have studied \emph{performative optimal} solutions that minimizes \eqref{perf}. This is challenging due to the non-convexity of ${\cal L}(\prm)$ and more importantly, the absence of knowledge of $\Pi_{\prm}$. In fact, evaluating $\grd {\cal L}(\prm)$ or its stochastic gradient estimate would require learning the distribution $\Pi_{\prm}$ \emph{a-priori} \citep{izzo2021learn}. To design a tractable procedure, prior works have assumed structures for \eqref{perf} such as approximating $\Pi_{\prm}$ by Gaussian mixture \citep{izzo2021learn}, $\Pi_{\prm}$ depends linearly on $\prm$ \citep{narang2022multiplayer}, etc., combined with a two-phase algorithm that separately learns $\Pi_{\prm}$ and optimizes $\prm$. Other works have assumed a \emph{mixture dominance} structure \citep{Miller2021OutsideTE} on the combined effect of $\Pi_{\prm}$ and $\ell(\cdot)$ on ${\cal L}(\prm)$, which in turn implies that ${\cal L}(\prm)$ is convex. Based on this assumption, a derivative free optimization (DFO) algorithm was analyzed in \cite{ray2022decision}; also see the variants of this condition when used in different settings \cite{wood2023stochastic, zhu2023online}.
On the other hand, several works have initiated the analysis of performative prediction with more general non-convex loss \cite{dong2023approximate, mofakhami2023performative}.

\begin{figure*}
  \centering
  \begin{minipage}{\textwidth}  
    \centering
    \resizebox{\linewidth}{!}{
    \begin{tabular}{l c c c c c l l}
\toprule
\bfseries Literature & \bfseries Cvx-$\ell$ & \bfseries Cvx-${\cal L}$ & \bfseries Dec-dep & \bfseries State-dep & \bfseries Oracle &\bfseries Rate & \bfseries $\theta_{\infty}$-Type 
\\
\midrule
\cite{ghadimi2013} & \cross & \cross & \cross & \cross & $0^{th}$ & $\mathcal{O}(T^{-\frac{1}{2}})$ & Stationary
\\
\cite{Miller2021OutsideTE} & \checkmark & \checkmark & \checkmark & \cross & $1^{st}{}^\dagger$ & ${\cal O}(T^{-1})$ & Perf.~Opt. \\
\cite{ray2022decision} & \checkmark & \checkmark & \checkmark & \checkmark (Geo)& $0^{th}$ & ${\mathcal{O}}(T^{-\frac{1}{2}})$ & Perf.~Opt.
\\
\cite{mendler2020stochastic} & \checkmark & \cross & \checkmark& \cross & $1^{st}$ & ${\mathcal{O}}(T^{-1})$ & Perf.~Stable 
\\
\cite{brown2022performative} & \checkmark & \cross & \checkmark& \checkmark (Geo) & $1^{st}$ & ${\mathcal{O}}(T^{-1})$ & Perf.~Stable
\\ 
\cite{li2022state} & \checkmark & \cross & \checkmark& \checkmark (Mkv) & $1^{st}$ & ${\mathcal{O}}(T^{-1})$ & Perf.~Stable
\\
\cite{izzo2021learn} & \cross & \cross & \checkmark& \checkmark (Geo) & $1^{st}{}^\dagger$ & ${\mathcal{O}}(T^{-\frac{1}{5}} 
)$ & Stationary
\\  
\cite{roy2022projection} & \cross & \cross &  \checkmark & \checkmark (Mkv) & $1^{st}{}^\ddagger$ & $\mathcal{O}(T^{-\frac{2}{5}})$ & Stationary 
\\
\cellcolor{gray!15} {\bf This Work} & \cellcolor{gray!15}\cross & \cellcolor{gray!15}\cross & \cellcolor{gray!15}\checkmark & \cellcolor{gray!15}\checkmark (Mkv) & \cellcolor{gray!15}$0^{th}$ &\cellcolor{gray!15}${\mathcal{O}}(T^{-\frac{1}{3} })$ & \cellcolor{gray!15}Stationary
\\
\bottomrule
\end{tabular}}
\captionof{table}{{\bf Comparison to Existing Literature}.
`$\theta_{\infty}$-Type' describes the convergent points -- \emph{Perf.~stable}: Def.~2 of Mendler et al., \emph{Perf.~opt.}: solves $\min_{\theta} {\cal L}(\theta)$, \emph{Stationary}: $\| \nabla {\cal L}(\theta) \| = 0$. `Geo'/`Mkv': geometric decay/Markov. The rates sometimes omit the logarithmic terms.
${}^\dagger$uses two-stages to estimate $\Pi_\theta$ with linear distribution model. ${}^\ddagger$needs asymp.~unbiased estimate of $\nabla {\cal L}(\theta)$ [cf.~(\ref{eq:all-gradient})]. 
}
    \label{tab:nested_table}
  \end{minipage}
\end{figure*}


This paper focuses on approximating the \emph{performative optimal} solution without relying on additional condition on the distribution $\Pi_{\prm}$ and/or using a two-phase algorithm. We concentrate on stochastic DFO algorithms \citep{ghadimi2013} which do not involve first order information (i.e., gradient) about ${\cal L}(\prm)$. As an advantage, these algorithms avoid the need for estimating $\Pi_{\prm}$ nor making structural assumption on the latter. Instead, the learner only requires access to the loss function evaluation oracle $\ell( \prm ; Z )$ and receive data samples from a controlled Markov chain. Note that the latter models the \emph{stateful} and \emph{strategic} agent setting considered in \citep{ray2022decision, roy2022projection, li2022state, brown2022performative}. 
{Such setting is motivated when the actual data distribution adapts slowly to the decision model, which is to be announced by the learner during the stochastic optimization process.}

The proposed {\algoname} algorithm features {\sf (i)} two-timescale step sizes design to control the bias-variance tradeoff in the derivative-free gradient estimates, and {\sf (ii)} a sample accumulation mechanism with forgetting factor $\lambda$ that aggregates every observed samples to control the amount of error in gradient estimates. Our findings are summarized as:
\begin{itemize}[leftmargin=*, nosep]
\item Under the Markovian data setting, we show in Theorem~\ref{thm1} that the {\algoname} algorithm finds a near-stationary solution $\bar{\prm}$ with $\EE[ \| \grd {\cal L}( \bar{\prm} ) \|^2] \leq \epsilon$ using ${\cal O}(\frac{d^2}{\epsilon^3}\log 1/\epsilon)$ samples/iterations. Compared to prior works, our analysis does not require structural assumption on the distribution $\Pi_{\prm}$ or convexity condition on the performative risk  \citep{izzo2021learn, Miller2021OutsideTE, ray2022decision}.
\item Our analysis demonstrates the trade-off induced by the forgetting factor $\lambda$. We identify the desiderata for the optimal value(s) of $\lambda$. 
We show that increasing $\lambda$ allows to reduce the number of samples requited by the algorithm if the performative risk gradient has a small Lipschitz constant.
\end{itemize}
For the rest of this paper, \S\ref{sec:setup} describes the problem setup and the {\algoname} algorithm, \S\ref{sec:main} presents the main results, \S\ref{sec:proof} outlines the proofs. Finally, we provide numerical results to verify our findings in \S\ref{sec:num}.

Finally, as displayed in Table~\ref{tab:nested_table}, for stationary points convergence of (\ref{perf}), our stochastic DFO under \emph{decision dependent} (and Markovian) samples has a convergence rate of ${\cal O}( 1/T^{\frac{1}{3}})$ towards an $\epsilon$-stationary point, while other known gradient-based methods, e.g., \citep{izzo2021learn} has a convergence rate of ${\cal O}(1/T^{\frac{1}{5}})$, and \citep{roy2022projection} has a convergence rate of ${\cal O}(1/T^{\frac{2}{5}})$. We also include other relevant results in the table.

We remark that our rate is worse than works on the decision independent (i.e., non-performative) setting, e.g., ${\cal O}(1/\sqrt{T})$ in \citep{ghadimi2013}. We believe the rate degradation is due to a fundamental limit for DFO-type algorithms when tackling problems with decision-dependent sample due to the challenges in designing a low variance gradient estimator (including 2-point estimators); see \S\ref{sec:hardness}. 




{\bf Related Works}.
The idea of DFO dates back to \cite{Nemirovski1983}, and has been extensively studied thereafter \cite{flaxman2005, Agarwal2010, Nesterov2017, ghadimi2013}. Results on matching lower bound were established in \citep{Jamieson2012}. 
While a similar DFO framework is adopted in the current paper for performative prediction, our algorithm is limited to using a special design in the gradient estimator to avoid introducing unwanted biases. 

Only a few works have considered the Markovian data setting in performative prediction. \cite{brown2022performative} is the first paper to study the dynamic settings, where the response of agents to learner's deployed classifier is modeled as a function of classifier and the current distribution of the {population}; also see \citep{izzo2022learn}. 
On the other hand, \citet{li2022state, roy2022projection} model the unforgetful nature and the reliance on past experiences of \emph{single/batch} agent(s) via controlled Markov Chain.
Lastly, \citet{ray2022decision} investigated the state-dependent framework where agents' response may be driven to best response at a geometric rate. The current paper considers more relaxed conditions than \citet{ray2022decision}. 
See Appendix~\ref{ass:lit} for a detailed discussion.







{\bf Notations}: Let $\mathbb{R}^d$ be the $d$-dimensional Euclidean space equipped with inner product $\langle\cdot, \cdot\rangle$ and induced norm $\|x\|=\sqrt{\langle x, x\rangle}$.
Let ${\sf Z}$ be a (measurable) sample space, ${\cal Z}$ be the Borel $\sigma$-algebra generated by ${\sf Z}$, and $\mu$, $\nu$ are two probability measures on ${\cal Z}$. Then, we use 
$\tv{\mu}{\nu} \eqdef \sup_{A\subset {\sf Z}}\mu(A)-\nu(A)$
to denote the total variation distance between $\mu$ and $\nu$
Denote $\TT_{\prm}(\cdot, \cdot)$ as the state-dependent Markov kernel and its stationary distribution is $\Pi_{\prm}(\cdot)$. Let $\mathbb{B}^{d}$ and $\mathbb{S}^{d-1}$ be the unit ball and {its boundary} (i.e., a unit sphere) centered around the origin in $d$-dimensional Euclidean space, respectively, and correspondingly, the ball and sphere of radius $r>0$ are $r\mathbb{B}^{d}$ and $r\mathbb{S}^{d-1}$.

\section{Problem Setup and Algorithm Design}
\label{sec:setup}

In this section, we develop the {\algoname} algorithm for tackling \eqref{perf} and describe the problem setup. 

Assume that ${\cal L}(\prm)$ is differentiable, we focus on finding an \emph{$\epsilon$-stationary} solution, $\prm$, which satisfies
\beq \label{eq:stationary}
\| \grd {\cal L} (\prm ) \|^2 \leq \epsilon.
\eeq 
With the goal of reaching \eqref{eq:stationary}, there are two key challenges in our stochastic algorithm design: {\sf (i)} to estimate the gradient $\grd {\cal L}(\prm)$ without prior knowledge of $\Pi_{\prm}$, and {\sf (ii)} to handle the \emph{stateful} setting where one cannot draw samples directly from the distribution $\Pi_{\prm}$. We shall discuss how the proposed {\algoname} algorithm, which is summarized in Algorithm~\ref{alg:dfo_lambda}, tackles the above issues through utilizing two ingredients: {\sf (a)} two-timescales step sizes, and {\sf (b)} sample accumulation with the forgetting factor $\lambda \in [0,1)$.

\begin{algorithm}[tb]
   \caption{{\algoname} Algorithm}
   \label{alg:dfo_lambda}
   
\begin{algorithmic}[1]
   \STATE {\bfseries Input:} Constants $\delta_{0}, \eta_{0}, \tau_0, \alpha, \beta$, maximum epochs $T$,
forgetting factor $\lambda,$ loss function $\ell\left( \cdot; \cdot \right)$.

    \STATE {\bfseries Initialization:} Set initial $\prm_0$ and sample $Z_0$.
   \FOR{$k=0 \textbf{ to } T-1$}
    \STATE \label{line:step} $\delta_{k}\leftarrow \delta_0/(1+k)^{\beta}$, 
    $\eta_{k}\leftarrow \eta_0/(1+k)^{\alpha}$,
    
    $\tau_{k} \leftarrow \max\{1, \tau_0\log(1+k)\}$
    \STATE Let $\prm_{k}^{(1)}\leftarrow\prm_{k}$, $Z_{k}^{(0)} \leftarrow Z_{k}$, 
    
    ${\bm u}_{k}\sim \unif(\mathbb{S}^{d-1})$
   \FOR{$m=1,2,\cdots, \tau_{k}$}
   \STATE Deploy the model $\check{\prm}_{k}^{(m)} = \prm_{k}^{(m)} + \delta_{k} {\bm u}_{k}$
    \STATE Draw $Z_{k}^{(m)} \sim \mathbb{ T}_{\cprm_{k}^{(m)}}(Z_{k}^{(m-1)}, \cdot)$
    \STATE Update $\prm_{k}^{(m)}$ as
    \begin{align*}
        &\textstyle {\bm g}_{k}^{(m)}  = \frac{d}{\delta_{k}} \ell \big(\cprm_{k}^{(m)}; Z_{k}^{(m)} \big) {\bm u}_{k},        \quad \prm_{k}^{(m+1)}  = \prm_{k}^{(m)}-\eta_{k}\lambda^{\tau_{k}-m} {\bm g}_{k}^{(m)}.
    \end{align*}   
    \ENDFOR
   \STATE $Z_{k+1}\leftarrow Z_{k}^{(\tau_{k})}$, $\prm_{k+1}\leftarrow\prm_{k}^{(\tau_{k}+1)}$.
   \ENDFOR

\STATE Draw $s\sim\text{Unif}\left(\{0,1,\ldots,T\}\right)$

{\bfseries Output:} Average iterate $\prm_{s}$,

\end{algorithmic}
\end{algorithm}

\textbf{Estimating $\grd {\cal L}(\prm)$ via Two-timescales DFO.}
First notice that the gradient of ${\cal L}(\cdot)$ can be derived as
\beq \label{eq:all-gradient}
\grd {\cal L}(\prm) = \EE_{ Z \sim \Pi_{\prm} } [ \grd \ell( \prm; Z) + \ell( \prm; Z ) \grd_{\prm} \log \Pi_{\prm}( Z ) ] .
\eeq 
As a result, constructing the stochastic estimates of $\grd {\cal L}(\prm)$ typically requires knowledge of $\Pi_{\prm}( \cdot )$ which may not be known a-priori unless a separate estimation procedure is applied; see e.g., \citep{izzo2021learn}. 
To avoid the need for direct evaluations of $\grd_{\prm} \log \Pi_{\prm}( Z )$, we consider an alternative design via zero-th order optimization \citep{ghadimi2013}. The intuition comes from observing that with $\delta \to 0^+$, $\left({\cal L}( \prm + \delta {\bm u} ) - {\cal L}( \prm )\right)/\delta$ is an approximation of the directional derivative of ${\cal L}$ along ${\bm u}$. This suggests that an estimate for $\grd {\cal L}(\prm)$ can be constructed using the \emph{objective function values} of $\ell(\prm;Z)$ only. 

Inspired by the above, we aim to construct a gradient estimate by querying $\ell(\cdot)$ at randomly perturbed points.
Formally, given the current iterate $\prm\in \RR^d$ and a query radius $\delta>0$, we sample a vector ${\bm u} \in \RR^{d}$ uniformly from $\mathbb{S}^{d-1}$.

The zero-th order gradient estimator for ${\cal L}(\prm)$ is then defined as
\begin{align}\label{eq:grd}
    g_{\delta}(\prm; {\bm u}, Z) \eqdef \frac{d}{\delta} \ell(\cprm; Z) \, {\bm u},~\text{with}~\cprm \eqdef \prm+\delta {\bm u}, 
\end{align}
and $Z \sim \Pi_{\cprm}(\cdot)$.
In fact, as ${\bm u}$ is zero-mean, $g_{\delta}( \prm; {\bm u}, Z )$ is an unbiased estimator for $\grd {\cal L}_{\delta} ( \prm )$. Here, ${\cal L}_{\delta} ( \prm )$ is a smooth approximation of ${\cal L}(\prm)$ \citep{flaxman2005, Nesterov2017} defined as
\beq 
{\cal L}_{\delta} ( \prm ) = \EE_{ {\bm u} } [ {\cal L} ( \cprm  ) ] = \EE_{ {\bm u} } [ \EE_{ Z \sim \Pi_{\cprm} } [ {\ell} ( \cprm ; Z ) ]  ].
\eeq 
It is known that under mild condition [cf.~Assumption \ref{assu:Lip} to be discussed later], $\| \grd {\cal L}_{\delta} ( \prm ) - \grd {\cal L}( \prm) \| = {\cal O}(\delta)$ 
and thus \eqref{eq:grd} is an ${\cal O}(\delta)$-biased estimate for $\grd {\cal L}(\prm)$. 

We remark that the gradient estimator in \eqref{eq:grd} differs from the one used in classical works on DFO such as \citep{ghadimi2013}. The latter takes the form of $\frac{d}{\delta} ( \ell(\cprm; Z) - \ell( \prm; Z ) ) \, {\bm u}$. Under the setting of standard stochastic optimization where the sample $Z$ is drawn \emph{independently} of ${\bm u}$ and Lipschitz continuous $\ell (\cdot; Z)$, the said estimator in \citep{ghadimi2013} is shown to have constant variance while it remains ${\cal O}(\delta)$-biased. Such properties \emph{cannot} be transferred to \eqref{eq:grd} since $Z$ is drawn from a distribution dependent on ${\bm u}$ via $\cprm = \prm + \delta {\bm u}$. In this case, the two-point gradient estimator would become biased; see \S\ref{sec:hardness}.

However, we note that the variance of \eqref{eq:grd} scales as ${\cal O}(1/\delta^2)$ when $\delta \to 0$, thus the parameter $\delta$ yields a bias-variance trade off in the estimator design. 
To remedy for the increase of variance, the {\algoname} algorithm incorporates a \emph{two-timescale step size} design for generating gradient estimates ($\delta_{k}$) and updating models ($\eta_{k}$), respectively. Our design principle is such that the models are updated at a \emph{slower timescale} to adapt to the gradient estimator with ${\cal O}(1/\delta^2)$ variance. Particularly, we will set $\eta_{k+1}/\delta_{k+1} \rightarrow 0$ to handle the bias-variance trade off, e.g., by setting $\alpha > \beta$ in line~4 of Algorithm~\ref{alg:dfo_lambda}.

\textbf{Markovian Data and Sample Accumulation.} We consider a setting where the sample/data distribution observed by the {\algoname} algorithm evolves according to a \emph{controlled Markov chain (MC)}. Notice that this describes a stateful agent(s) scenario such that the deployed models ($\prm$) would require a certain amount of time to manifest their influence on the samples obtained; see \citep{li2022state, roy2022projection, brown2022performative, ray2022decision, izzo2022learn}. 

To formally describe the setting, we denote $\TT_{\prm}: {\sf Z} \times {\cal Z} \to \RR_+$ as a Markov kernel controlled by a deployed model $\prm$. For a given $\prm$, the kernel has a unique stationary distribution $\Pi_{\prm}(\cdot)$ which cannot be conveniently accessed. Under this setting, suppose that the previous state/sample is $Z$, the next sample follows the distribution $Z' \sim \TT_{\prm} ( Z, \cdot )$ which is not necessarily the same as $\Pi_{\prm}(\cdot)$.
As a consequence, the gradient estimator \eqref{eq:grd} is a biased estimator of $\grd {\cal L}_{\delta} (\prm)$.

A common strategy in settling the above issue is to allow a \emph{burn-in} phase in the algorithm as in \citep{ray2022decision}; also commonly found in MCMC methods \citep{robert1999monte}. Using the fact that $\TT_{\prm}$ admits the stationary distribution $\Pi_{\prm}$, if one can wait a sufficiently long time before applying the current sample, i.e., consider initializing with the previous sample $Z^{(0)} = Z$, the procedure\vspace{-.0cm}
\beq \label{eq:burnin}
Z^{(m)} \sim \TT_{\prm} (Z^{(m-1)}, \cdot),~m=1,\ldots,\tau,\vspace{-.2cm}
\eeq  
would yield a sample $Z^+ = Z^{(\tau)}$ that admits a distribution close to $\Pi_{\prm}$ provided that $\tau \gg 1$ is sufficiently large compared to the mixing time of $\TT_{\prm}$.

Intuitively, the procedure \eqref{eq:burnin} can be inefficient as a number of samples $Z^{(1)}, Z^{(2)}, \ldots, Z^{(\tau-1)}$ will be completely ignored at the end of each iteration.
As a remedy, the {\algoname} algorithm incorporates a sample accumulation mechanism which gathers the gradient estimates generated from possibly non-stationary samples via a forgetting factor of $\lambda \in [0,1)$. {Following \eqref{eq:grd},  $\grd {\cal L}(\prm)$ is estimated by
\beq \textstyle 
{\bm g} = \frac{d}{\delta} \sum_{m=1}^{\tau} \lambda^{\tau - m} \ell( \prm^{(m)} + \delta {\bm u} ; Z^{(m)} ) \, {\bm u}, 
\eeq
with $Z^{(m)} \sim \mathbb{ T}_{ \prm^{(m)} + \delta {\bm u} }(Z^{(m-1)}, \cdot).$

At a high level, the mechanism works by assigning large weights to samples that are close to the end of an epoch (which are less biased). Moreover, $\prm^{(m)}$ is \emph{simultaneously updated} within the epoch} to obtain an online algorithm that gradually improves the objective value of \eqref{perf}. Note that with $\lambda=0$, the {\tt DFO}(0) algorithm reduces into one that utilizes \emph{burn-in} \eqref{eq:burnin}. 
We remark that from the implementation perspective for performative prediction, Algorithm~\ref{alg:dfo_lambda} corresponds to a \emph{greedy deployment} scheme \citep{mendler2020stochastic} as the latest model {$\prm_{k}^{(m)} + \delta_{k} {\bm u}_{k}$} is deployed at every sampling step. Line 6--10 of Algorithm~\ref{alg:dfo_lambda} details the above procedure. 

Lastly, we note that recent works have analyzed stochastic algorithms that rely on a \emph{single trajectory} of samples taken from a Markov Chain, e.g., \citep{sun2018markov,  karimi2019non, doan2022finite}, that are based on stochastic gradient. 
\cite{sun2019decentralized} considered a {\tt DFO} algorithm for general optimization problems but the MC studied is not controlled by $\prm$. 
\vspace{-.2cm}

\section{Main Results}\label{sec:main}
\vspace{-.1cm} This section studies the convergence of the {\algoname} algorithm and demonstrates that the latter finds an $\epsilon$-stationary solution [cf.~\eqref{eq:stationary}] to (\ref{perf}). 

We first state the assumptions required for our analysis:
\begin{assumption}{\bf (Smoothness)}\label{assu:Lip}
${\cal L}(\prm)$ is differentiable, and there exists a constant $L > 0$ such that  
\[
    \norm{\grd {\cal L}(\prm) - \grd {\cal L}(\prm^\prime)} \leq L\norm{\prm-\prm^\prime}, ~ \forall \prm, \prm^\prime \in \RR^d.
\]
\end{assumption}

\begin{assumption}{\bf (Bounded Loss)}\label{assu:BoundLoss} There exists a constant $\G > 0$ such that
\[
     |\ell(\prm; z)| \leq \G,~\forall~ \prm\in \RR^d, ~\forall ~z\in {\sf Z}.
\]
\end{assumption}
\begin{assumption}{\bf (Lipschitz Distribution Map)}\label{assu:smooth_dist}
There exists a constant $L_1 > 0$ such that
\[
    \tv{\Pi_{\prm_1}} {\Pi_{\prm_2}}\leq L_1 \norm{\prm_1-\prm_2}\quad\forall\prm_1,\prm_2\in\mathbb{R}^d.
\]
\end{assumption}
Informally, the conditions above state that the gradient of the performative risk is Lipschitz continuous and the state-dependent distribution vary smoothly w.r.t.~$\prm$. Note that Assumption~\ref{assu:Lip} and Assumption~\ref{assu:BoundLoss} are both regularity conditions that can also be found in \citep{izzo2021learn, ray2022decision}. Assumption~\ref{assu:smooth_dist} is slightly strengthened from the Wasserstein-1 distance bound in \citep{perdomo2020performative}, and it gives tighter control for distribution shift in our Markovian data setting. Note that this particular condition can be slightly relaxed, see Lemma~\ref{lem:lip_decoupled_risk_wasserstein} in the appendix.

Next, we consider the assumptions about the controlled Markov chain induced by $\TT_{\prm}$:
\begin{assumption}{\bf (Geometric Mixing)}\label{assu:FastMixing}
Let $\{Z_{k}\}_{k\geq 0}$ denote a Markov Chain on the state space ${\sf Z}$ with transition kernel $\TT_{\prm}$ and stationary measure $\Pi_{\prm}$. 
There exist constants $\rho \in [0,1)$, $M \geq 0$, such that for any $k\geq 0$, $z\in {\sf Z}$,
\[
    \tv{\mathbb{ P}_{\prm}(Z_k\in \cdot | Z_0 =z)}{\Pi_{\prm}} \leq M  \rho^{k}.
\]
\end{assumption}
\begin{assumption}{\bf (Smoothness of Markov Kernel)}\label{assu:smooth_kernel}
There exists a constant $L_{2} \geq 0$ such that 
\begin{align*}
    \tv{\TT_{\prm_1}(z, \cdot) }{\TT_{\prm_2}(z, \cdot)} \leq L_{2} \norm{\prm_1-\prm_2},
\end{align*}
holds for any $\prm_1, \prm_2\in \RR^{d},~z\in {\sf Z}.$

\end{assumption}
Assumption \ref{assu:FastMixing} is a standard condition on the mixing time of the Markov chain induced by $\TT_{\prm}$;
Assumption \ref{assu:smooth_kernel} imposes a smoothness condition on the Markov transition kernel $\TT_{\prm}$ with respect to $\prm$. The geometric dynamically environment in \cite{ray2022decision} constitutes a special case which satisfies the above conditions, yet Assumption \ref{assu:FastMixing} is strictly more general.

Unlike \citep{ray2022decision, izzo2021learn, Miller2021OutsideTE}, we do not impose any additional assumption on $\Pi_{\prm}$ such as mixture dominance other than Assumption~\ref{assu:smooth_dist}. As a result, \eqref{perf} remains an `unstructured' non-convex stochastic optimization problem with decision-dependent distribution. 
Our main result on the convergence of the {\algoname} algorithm towards a near-stationary solution of \eqref{perf} is summarized as:
\begin{theorem}\label{thm1}
Suppose Assumptions \ref{assu:Lip}-\ref{assu:smooth_kernel} hold, step size sequence $\{\eta_{k}\}_{k\geq 1}$, and query radius sequence $\{\delta_{k}\}_{k\geq 1}$ satisfy the following conditions,
\beq \label{eq:stepsize_thm}
\begin{aligned}
&\eta_{k} = d^{-2/3} \cdot (1+k)^{-2/3}, \quad \delta_{k} = d^{1/3} \cdot (1+k)^{-1/6},
\\
&\tau_{k}=\max\{1,\frac{2}{\log 1/\max\{\rho,\lambda\}}\log (1+k)\}\quad \forall k\geq 0.
\end{aligned}
\eeq 
Then, there exists constants $t_0$, $c_5, c_6, c_7$, such that for any $T\geq t_0$, the iterates $\{\prm_k\}_{k\geq 0}$ generated by \algoname~satisfy the following inequality,
\beq \label{eq:main_result}
\begin{aligned}
     &\frac{1}{1+T}\sum_{k=0}^{T} \EE\norm{\grd {\cal L}(\prm_k)}^2 
     \leq 12\max\left\{c_5 (1-\lambda), c_6, \frac{c_7}{1-\lambda}\right\} \frac{ d^{2/3}}{(T+1)^{1/3}}.
\end{aligned}
\eeq 
\end{theorem}\noindent
We have defined the following quantities and constants: $c_5=2 \G$,
\begin{align} \label{eq:c567}
c_6=\frac{\max\{L^2,\G^2(1 - \beta)\}}{ 1-2\beta },
    ~  c_7=\frac{L\G^2 }{ 2 \beta - \alpha + 1 },
\end{align}
with ${\alpha = \frac{2}{3}, \beta = \frac{1}{6}}$. Observe the following corollary on the iteration complexity of {\algoname} algorithm:
\begin{Corollary} ($\epsilon$-stationarity)
Suppose that the Assumptions of Theorem \ref{thm1} hold. Fix any $\epsilon>0$, the condition $\frac{1}{1+T} \sum_{k=0}^{T}\EE\norm{\grd {\cal L}(\prm_k)}^2 \leq \epsilon$ holds whenever
\beq 
    T \geq \left( 12 \max\left\{c_5 (1-\lambda), c_6, \frac{c_7}{1-\lambda}\right\} \right)^3 \frac{d^2}{\epsilon^3} .
\eeq 
\end{Corollary}
In the corollary above, the lower bound on $T$ is expressed in terms of the number of epochs that Algorithm \ref{alg:dfo_lambda} needs to achieve the target accuracy. Consequently, the total number of samples required (i.e., the number of inner iterations taken in Line 6--9 of Algorithm~\ref{alg:dfo_lambda} across all epochs) is: 
\beq \label{eq:sample-complex} 
{\tt S}_{\epsilon} = \sum_{k=1}^{T} \tau_k = {\cal O}\left( \frac{d^2}{\epsilon^3} \log(1/\epsilon)\right).
\eeq 

We remark that due to the decision-dependent properties of the samples, the {\algoname} algorithm exhibits a worse sampling complexity \eqref{eq:sample-complex} than prior works in stochastic DFO algorithm, e.g., \citep{ghadimi2013} which shows a rate of ${\cal O}( d / \epsilon^2 )$ on non-convex smooth objective functions.
In particular, the adopted one-point gradient estimator in \eqref{eq:grd} admits a variance that can only be controlled by a time varying $\delta$; see the discussions in \S\ref{sec:hardness}. 
 
Achieving the desired convergence rate requires setting $\eta_k = \Theta( k^{-2/3} )$, $\delta_k = \Theta( k^{-1/6} )$, i.e., yielding a two-timescale step sizes design with $\eta_k / \delta_k \to 0$. Notice that the influence of forgetting factor $\lambda$ are reflected in the constant factor of \eqref{eq:main_result}. Particularly, if $c_5 > c_7$ and $c_5\geq c_6$, the optimal choice is $\lambda = 1-\sqrt{\frac{c_7}{c_5}}$, otherwise the optimal choice is $\lambda\in\left[0, 1-c_7/c_6\right]$. Informally, this indicates that when the performative risk is smoother (i.e. its gradient has a small Lipschitz constant), a large $\lambda$ can speed up the convergence of the algorithm; otherwise a smaller $\lambda$ is preferable.

\section{Proof Outline of Main Results}\label{sec:proof}
This section outlines the key steps in proving Theorem~\ref{thm1}. Notice that analyzing the {\algoname} algorithm is challenging due to the two-timescales step sizes and Markov chain samples with time varying kernel. Our analysis departs significantly from prior works such as \citep{ray2022decision, izzo2021learn, brown2022performative, li2022state} to handle the challenges above. 

Let ${\cal F}^k = \sigma( \prm_0, Z_{s}^{(m)}, u_s, 0 \leq s \leq k, 0 \leq m \leq \tau_{k})$ be the filtration. 
Our first step is to exploit the smoothness of ${\cal L}(\prm)$ to bound the squared norms of gradient. Observe that:
\begin{lemma}{\bf (Decomposition)}
\label{lem:decomposition}
Under Assumption \ref{assu:Lip}, 
it holds that
\begin{equation} \label{eq:decompose1}
\sum_{k=0}^{t}\EE\norm{\grd {\cal L}(\prm_k)}^2 \leq \term{1} + \term{2} + \term{3} + \term{4},
\end{equation}
for any $t \geq 1$, where 
\begin{align*}
&\textstyle \term{1} \eqdef \sum_{k=1}^{t}\frac{1-\lambda}{\eta_{k}}\left(\EE\left[{\cal L}(\prm_k)\right]-\EE\left[{\cal L}(\prm_{k+1})\right]\right)
\\
    &\textstyle \term{2} \eqdef -\sum_{k=1}^{t}\EE\bigg\langle \nabla {\cal L}(\prm_k) 
    \big| 
    (1-\lambda)\sum_{m=1}^{\tau_{k}}\lambda^{\tau_{k}-m} 
    \cdot \left(g_{k}^{(m)}-\EE_{Z\sim \Pi_{\cprm_k}}\left[g_{\delta_k}(\prm_{k};u_k, Z)\right]\right) \bigg\rangle
\\
    & \textstyle \term{3} \eqdef -\sum_{k=1}^{t}\EE\bigg\langle \grd {\cal L}(\prm_k) {\big|}  (1-\lambda)     \left(\sum_{m=1}^{\tau_{k}}\lambda^{\tau_{k}-m}\grd{{\cal L}_{\delta_{k}}(\prm_{k})}\right)-\grd {\cal L}(\prm_k) \bigg\rangle
\\
    & \textstyle \term{4} \eqdef \frac{L(1-\lambda)}{2}\sum_{k=1}^{t}\eta_{k}\EE \norm{\sum_{m=1}^{\tau_{k}}\lambda^{\tau_{k}-m}g_{k}^{(m)}}^2
\end{align*}
\end{lemma}
The lemma is achieved through the standard descent lemma implied by Assumption~\ref{assu:Lip} and decomposing the upper bound on $|| \grd {\cal L}(\prm_k) ||^2$ into respectful terms; see the proof in Appendix~\ref{ap:decompose}. Among the terms on the right hand side of \eqref{eq:decompose1}, $\term{1},\term{3}$ and $\term{4}$ arises directly from Assumption~\ref{assu:Lip}, while $\term{2}$ comes from bounding the noise terms due to Markovian data.

We bound the four components in Lemma \ref{lem:decomposition} as follows. For simplicity, we denote ${\cal A}(t) \eqdef \frac{1}{1+t} \sum_{k=0}^{t} \EE\norm{\grd {\cal L}(\prm_k)}^2$.
Among the four terms, we highlight that the main challenge lies on obtaining a tight bound for $\term{2}$. Observe that\vspace{-.1cm}
\beq\label{eq:aa}
\begin{aligned}
\frac{\term{2}}{1-\lambda} & \leq \EE \left[ \sum_{k=0}^{t} \norm{\grd {\cal L}(\prm_k)} \bigg\|\sum_{m=1}^{\tau_{k}}\lambda^{\tau_{k} -m} \Delta_{k,m} \bigg\| \right], \\
\Delta_{k,m} & \overset{\text{def}}{=} \EE_{{\cal F}^{k-1}}\![ g_{k}^{(m)}\!-\!\EE_{Z\sim \Pi_{\cprm_k}}g_{k}(\prm_{k};u_{k},Z)]. \vspace{-.1cm}
\end{aligned}
\eeq
There are two sources of bias in $\Delta_{k,m}$: one is the noise induced by drifting of decision variable within each epoch, the other is the bias that depends on the mixing time of Markov kernel. 

To control these biases, we introduce a reference Markov chain $\Tilde{Z}_{k}^{(\ell)}$, $\ell=0,...,\tau_k$, whose decision variables remains fixed for a period of length $\tau_k$ and is initialized with $\Tilde{Z}_{k}^{(0)} = Z_{k}^{(0)}$:
\vspace{-.1cm}
\begin{align}
    &\Tilde{Z}_{k}^{(0)}\xrightarrow[]{\cprm_{k}}\Tilde{Z}_{k}^{(1)}\xrightarrow[]{\cprm_{k}}\Tilde{Z}_{k}^{(2)}\xrightarrow[]{\cprm_{k}}\Tilde{Z}_{k}^{(3)}\cdots\xrightarrow[]{\cprm_{k}}\Tilde{Z}_{k}^{(\tau_{k})}
    \label{chain:ref}
\end{align}
and we recall that the actual chain in the algorithm evolves as
\begin{equation}
    Z_{k}^{(0)}\xrightarrow[]{\cprm_{k+1}^{(0)}}Z_{k}^{(1)}\xrightarrow[]{\cprm_{k+1}^{(1)}}Z_{k}^{(2)}\cdots\xrightarrow[]{\cprm_{k+1}^{(\tau_{k}-1)}}Z_{k}^{(\tau_{k})}.
    \label{chain:real}
\end{equation}
Note the reference Markov chain idea is inspired by {\citep[Theorem 4.7]{wu2020finite}} which studied the convergence of a strongly convex subproblem.


With the help of the reference chain, we decompose the conditional expectation $\Delta_{k,m}$ as:
\begin{align*}
    \Delta_{k, m} &= \EE_{{\cal F}^{k-1}}
    \bigg[\frac{d}{\delta_{k}}  \bigg(\EE[\ell(\cprm_{k}^{(m)}; Z_{k}^{(m)})|\cprm_{k}^{(m)}, Z_{k}^{(0)}]
     - \EE_{\Tilde{Z}_{k}^{(m)}}[\ell(\cprm_{k}^{(m)}; \Tilde{Z}_{k}^{(m)})|\cprm_{k}^{(m)}, \Tilde{Z}_{k}^{(0)}] \bigg) {\bm u}_{k}\bigg]
    \\
    &\quad + \EE_{{\cal F}^{k-1}} \bigg[ \frac{d}{\delta_{k}}   \bigg(\EE_{\Tilde{Z}_{k}^{(m)}}[\ell(\cprm_{k}^{(m)}; \Tilde{Z}_{k}^{(m)})|\cprm_{k}^{(m)}, \Tilde{Z}_{k}^{(0)}]- \EE_{Z\sim \Pi_{\cprm_k}}[\ell(\cprm_{k}^{(m)}; Z)|\cprm_{k}^{(m)}] \bigg) {\bm u}_{k} \bigg]
    \\
    &\quad + {\EE}_{{\cal F}^{k-1}} \Big[\frac{d}{\delta_{k}} \underset{Z\sim\Pi_{\cprm_{k}}}{\EE} \left[ \ell(\cprm_{k}^{(m)}; Z) - \ell(\cprm_k; Z)|\cprm_{k}^{(m)}, \cprm_{k}\right]  {\bm u}_{k}\Big] 
        \\
        & \eqdef A_{1} + A_{2} + A_{3}
\end{align*}
We remark that $A_1$ reflects the drift of (\ref{chain:real}) from initial sample $Z_k^{(0)}$ driven by varying $\check{\prm}_{k}^{(m)}$, 
$A_2$ captures the statistical discrepancy between above two Markov chains (\ref{chain:real}) and (\ref{chain:ref}) at same step $m$, and $A_3$ captures the drifting gap between $\check{\prm}^{}_{k}$ and $\check{\prm}^{(m)}_{k}$. Applying Assumption \ref{assu:smooth_dist}, $A_1$ and $A_2$ can be upper bounded with the smoothness and geometric mixing property of Markov kernel. In addition, $A_3$ can be upper bounded using Lipschitz condition on (stationary) distribution map $\Pi_{\prm}$. 
Finally, the forgetting factor $\lambda$ helps to control $\normtxt{\check{\prm}_{k}^{(\cdot)} - \check{\prm}_{k} }$ to be at the same order of a single update. Therefore, $\norm{\Delta_{k,m}}$ can be controlled by an upper bound relying on $\lambda, \rho, L$.


The following lemma summarizes the above results as well as the bounds on the other terms:
\begin{lemma}\label{lem:bound_four_terms} 
Under Assumption  \ref{assu:BoundLoss}, \ref{assu:smooth_dist}, \ref{assu:FastMixing} and \ref{assu:smooth_kernel}, with $\eta_{t+1}=\eta_{0}(1+t)^{-\alpha}$, $\delta_{t+1}=\delta_0 (1+t)^{-\beta}$ and $\alpha \in (0,1)$, $\beta \in (0, \frac{1}{2})$. Suppose $0 < 2\alpha-4\beta < 1$ and
\[ 
\tau_{k}\geq\frac{1}{\log 1/\max\{\rho,\lambda\}}\left(\log(1+k)+\max\{\log\frac{\delta_0}{d}, 0\}\right).
\] 
Then, it holds that $\forall ~t \geq \max\{t_1,t_2\}$
\begin{align}
    {\bf I}_{1}(t) &\leq c_1 (1-\lambda) (1+t)^{\alpha},
    \\
    {\bf I}_{2}(t) & \leq \frac{ c_2 d^{5/2}}{(1-\lambda)^2}{{\cal A}(t)}^{\frac{1}{2}} (1+t)^{1-(\alpha-2\beta)},  \label{eq:i2}
    \\
    {\bf I}_{3}(t) &\leq c_3 {{\cal A}(t)}^{\frac{1}{2}} (1+t)^{1-\beta}, 
    \\
    {\bf I}_{4}(t) &\leq  \frac{c_4 d^2}{1-\lambda} (1+t)^{1-\left(\alpha-2\beta\right)}, \label{eq:i4-bound}
\end{align}
where $t_1, t_2$ are defined in \eqref{eq:const-t1}, \eqref{eq:const-t2}, and $c_1, c_2, c_3, c_4$ are constants defined as follows:
\begin{align*}
    c_{1} &\eqdef 2\G / \eta_0,
    ~~
    c_{2} \eqdef\frac{\eta_0}{\delta_0^2}\frac{6\cdot (L_1\G^2+L_2\G^2+\sqrt{L}\G^{3/2} )}{\sqrt{1-2\alpha+4\beta}},
    \\
    \textstyle c_{3} &\eqdef \frac{2}{\sqrt{1-2\beta}}\max\{L\delta_0, \G \sqrt{1-\beta}\},
    \\
    c_{4} &\eqdef \frac{\eta_{0}}{\delta_{0}^2}\cdot\frac{ L \G^2}{2\beta-\alpha+1}.
\end{align*}
\end{lemma}
See Appendix \ref{App_bound_four_terms} for the proof. 

We comment that the bound for ${\bf I}_{4}(t)$ cannot be improved. As a concrete example, consider the constant function
$\ell(\prm; z) = {\rm c} \neq 0$ for all $z \in {\sf Z}$, it can be shown that $\| g_{k}^{(m)} \|^2 = {\rm c}^2$ and consequently ${\bf I}_4(t) = \Omega( \eta_{k}/ \delta_{k}^2 ) = \Omega( t^{1-(\alpha-2\beta)} )$, which matches \eqref{eq:i4-bound}.

Finally, plugging Lemma~\ref{lem:bound_four_terms} into Lemma~\ref{lem:decomposition} gives:
\begin{align*}
    {\cal A}(t) &\leq \frac{c_1 (1-\lambda)}{(1+t)^{1-\alpha}} + \frac{ c_2 d^{5/2}}{(1-\lambda)^2} \frac{ {{\cal A}(t)}^{\frac{1}{2}} }{ (1+t)^{\alpha-2\beta} }  + c_3 \frac{ {{\cal A}(t)}^{\frac{1}{2}} }{ (1+t)^\beta } + c_4 \frac{d^2}{1-\lambda} \frac{1}{(1+t)^{\alpha-2\beta} }.
\end{align*}
Since ${\cal A}(t) \geq 0$, the above is a quadratic inequality that implies the following bound:
\begin{lemma}
\label{lem:major_bound}
    Under Assumption \ref{assu:Lip}--\ref{assu:smooth_kernel}, with the step sizes $\eta_{t+1}=\eta_{0}(1+t)^{-\alpha}$, $\delta_{t+1}=\delta_0 (1+t)^{-\beta}, \tau_{k}\geq\frac{1}{\log 1/\max\{\rho,\lambda\}}\left(\log(1+k)+\max\{\log\frac{\delta_0}{d}, 0\}\right)$, $\eta_0=d^{-2/3}, \delta_0=d^{1/3}$, $\alpha \in (0,1)$, $\beta \in (0, \frac{1}{2})$. If $2\alpha-4\beta < 1$, then there exists a constant $t_0$ such that the iterates $\{\prm_k\}_{k\geq 0}$ satisfies the following inequality for all $T \geq t_0$,
    \begin{align*}
        &\frac{1}{1+T}\sum_{k=0}^{T}\EE\norm{\grd{{\cal L}(\prm_k)}}^2 \leq 12 \max \left\{c_5 (1-\lambda), c_6, \frac{c_7}{1-\lambda} \right\} d^{2/3} T^{-\min\{2\beta,1-\alpha,\alpha-2\beta\}}.
    \end{align*}
\end{lemma}
Optimizing the step size exponents $\alpha, \beta$ in the above concludes the proof of Theorem~\ref{thm1}.



\subsection{Discussions}\label{sec:hardness}
We conclude by discussing two alternative zero-th order gradient estimators to \eqref{eq:grd}, and argue that they do not improve over the sample complexity in the proposed {\algoname} algorithm. We study:
\begin{align*}\textstyle
    {\bm g}_{\sf 2pt-I}  &\eqdef \frac{d}{\delta}  \left[\ell\left(\prm+\delta {\bm u}; Z \right) - \ell(\prm; Z) \right] {\bm u}, 
    \\
    {\bm g}_{\sf 2pt-II}  &\eqdef \frac{d}{\delta}  \left[\ell\left(\prm+\delta {\bm u}; Z_1 \right) - \ell(\prm; Z_2) \right] {\bm u}, 
\end{align*}
where ${\bm u} \sim \unif(\mathbb{S}^{d-1})$. For ease of illustration, we assume that the samples $Z, Z_1, Z_2$ are drawn directly from the stationary distributions $Z \sim \Pi_{\prm+\delta u}$, $Z_1 \sim \Pi_{\prm+\delta {\bm u} }$, $Z_2\sim \Pi_{\prm}$. 

We recall from \S\ref{sec:setup} that the estimator ${\bm g}_{\sf 2pt-I}$ is a finite difference approximation of the directional derivative of objective function along the randomized direction ${\bm u}$\footnote{In \cite{Nesterov2017, ghadimi2013}, ${\bm u}$ is drawn from the Gaussian distribution.}, as proposed in \cite{Nesterov2017, ghadimi2013}. For non-convex stochastic optimization with decision independent sample distribution, i.e., $\Pi_{\prm} \equiv \bar{\Pi}$ for all $\prm$, the DFO algorithm based on ${\bm g}_{\sf 2pt-I}$ is known to admit an optimal sample complexity of ${\cal O}(1/\epsilon^2)$ \citep{Jamieson2012}.
Note that $\EE_{{\bm u} \sim {\unif(\mathbb{S}^{d-1})} , Z \sim \bar{\Pi}} [ \ell(\prm;Z) {\bm u} ] = {\bm 0}$.

However, in the case of decision-dependent sample distribution as in \eqref{perf}, ${\bm g}_{\sf 2pt-I}$ would become a \emph{biased} estimator since the sample $Z$ is drawn from $\Pi_{\prm+\delta {\bm u}}$ which depends on ${\bm u}$. The DFO algorithm based on ${\bm g}_{\sf 2pt-I}$ may not converge to a stationary solution of \eqref{perf}. 

A remedy to handle the above issues is to consider the estimator ${\bm g}_{\sf 2pt-II}$ which utilizes \emph{two samples} $Z_1, Z_2$, each independently drawn at a different decision variable, to form the gradient estimate. In fact, it can be shown that $\EE[ {\bm g}_{\sf 2pt-II} ] = \grd {\cal L}_\delta( \prm )$ yields an unbiased gradient estimator. 
However, due to the decoupled random samples $Z_1, Z_2$, we have
\begin{align*}
    {\textstyle \EE\norm{{\bm g}_{\sf 2pt-II}}^2} &= \frac{d^2}{\delta^2} \EE\left(\ell\left(\prm+\delta {\bm u}; Z_1 \right) - \ell(\prm; Z_1) + \ell(\prm; Z_1) - \ell(\prm; Z_2)\right)^2\\
    &\overset{(a)}{\geq} \frac{d^2}{\delta^2} \EE\bigg[\frac{3}{4}\left(\ell(\prm; Z_1) - \ell(\prm; Z_2)\right)^{2}  - 3\left(\ell\left(\prm+\delta {\bm u}; Z_1 \right) - \ell(\prm; Z_1)\right)^{2}\bigg] \\
    &= \frac{3d^2}{\delta^2} \!\! \left( \frac{{\sf Var}[\ell(\prm;Z)]}{2} - \EE\left[\left(\ell\left(\prm+\delta {\bm u}; Z_1 \right) - \ell(\prm; Z_1)\right)^{2}\right] \right)
    \\
    &\overset{(b)}{\geq} \frac{3}{2}\frac{\sigma^2 d^2}{\delta^2}-3\mu^2 d^2 = \Omega(d^2 /\delta^2).
\end{align*}
where in (a) we use the fact that $(x+y)^2\geq\frac{3}{4}x^2-3 y^2$, in (b) we assume ${\sf Var}[\ell(\prm;Z)]\eqdef\EE\left(\ell(\prm;Z)-{\cal L}(\prm)\right)^2\geq \sigma^2>0$ and $\ell(\prm;z)$ is $\mu$-Lipschitz in $\prm$.
As such, this two-point gradient estimator does not significantly reduce the variance when compared with \eqref{eq:grd}. Note that a two-sample estimator also incurs additional sampling overhead in the scenario of Markovian sampling.

\vspace{-.2cm}
\section{Numerical Experiments}\label{sec:num}\vspace{-.2cm}
\begin{figure*}[th]
    \centering
    \includegraphics[width=.29\linewidth]{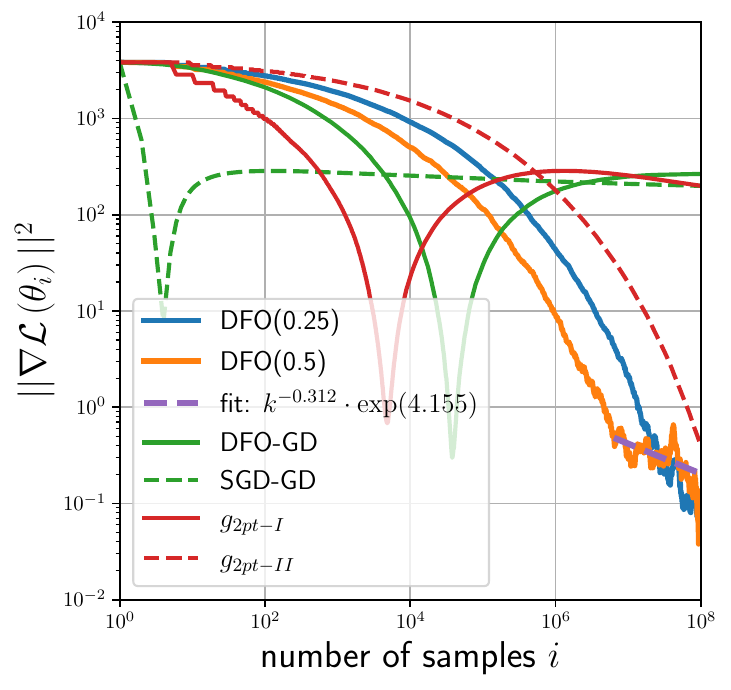}
    \includegraphics[width=.29\linewidth]{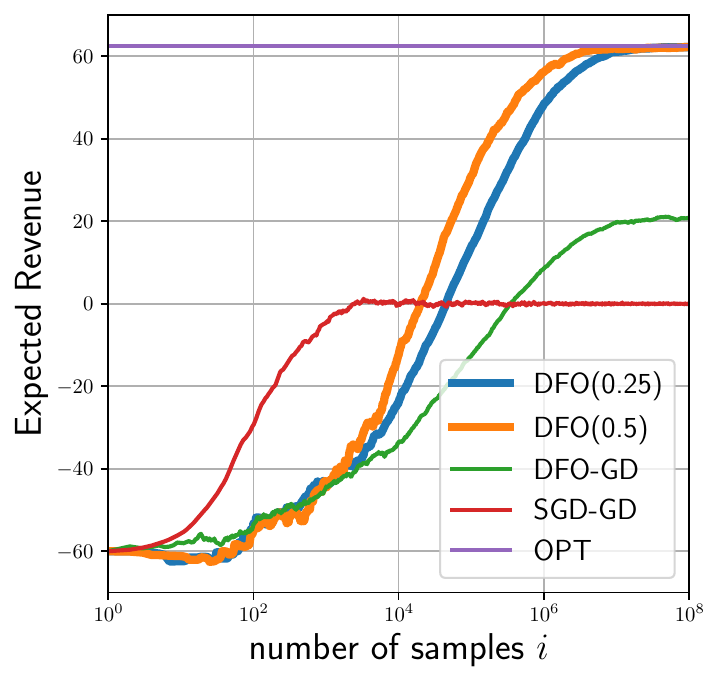}
    \includegraphics[width=.29\linewidth]{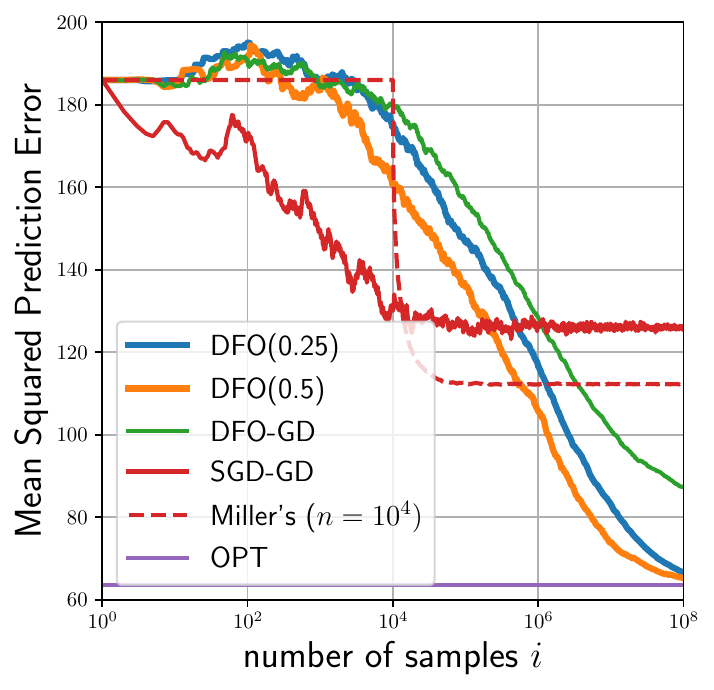}
    \caption{(\emph{left}) One Dimension Quartic Minimization problem with samples generated by AR distribution model where regressive parameter $\gamma=0.5.$ (\emph{middle}) Markovian Pricing Problem with $d=5$ dimension. (\emph{right}) Linear Regression problem based on AR distribution model ($\gamma=0.5$). }
    \label{fig:simulation}
\end{figure*}

We examine the efficacy of the {\algoname} algorithm on a few examples by comparing {\algoname} to several benchmarks. Unless otherwise specified, we use the step size choices in \eqref{eq:stepsize_thm} for {\algoname}. {All experiments are conducted on a server with an Intel Xeon 6318 CPU using Python 3.7.}\vspace{.1cm}
The expected performance are estimated using at least $10$ trials.

\textbf{1-Dimensional Case: Quartic Loss.}
The first example considers a scalar polynomial loss function $\ell : \RR \times \RR \to \RR$ defined by $\ell(\prm; z)=\frac{1}{12} z \prm (3\prm^2-8\prm-48)$. To simulate the controlled Markov chain scenario, the samples are generated dynamically according to an auto-regressive (AR) process $Z_{t+1}=(1-\gamma) Z_t + \gamma \bar{Z}_{t+1}$ with $\bar{Z}_{t+1} \sim {\cal N}(\prm, \frac{(2-\gamma)}{\gamma}\sigma^2)$ with parameter $\gamma\in (0,1)$. 
Note that the stationary distribution of the AR process is $\Pi_{\prm}={\cal N}(\prm, \sigma^2)$. As such, the performative risk function in this case is ${\cal L}(\prm)=\EE_{Z\sim\Pi_{\prm}}\left[\ell(\prm;Z)\right]=\frac{\prm^2}{12}(\prm^2-8\prm-48)$, which is quartic in $\prm$. 
Note that ${\cal L}(\prm)$ is not convex in $\prm$ and the set of stationary solution is $\{ \prm: \grd {\cal L}(\prm) = 0 \} = \{-2, 0, 4\}$, among which the optimal solution is $\prm_{PO} = \argmin_{ \prm } {\cal L}( \prm ) = 4$. 

In our experiments below, all the algorithms are initialized by $\prm_0= 6$.
In Figure \ref{fig:simulation} (left), we compare the norms of the gradient for performative risk with {\tt DFO}-{\tt GD} (no burn-in phase), the {\tt DFO}($\lambda$) algorithm, stochastic gradient descent with greedy deployment scheme (SGD-GD), and two 2-point estimators, ${\bm g}_{\sf 2pt-I},{\bm g}_{\sf 2pt-II}$ (with a burn-in phase) as described in Section \ref{sec:hardness} against the number of samples observed by  the algorithms.
We first observe from Figure \ref{fig:simulation} (left) that {\tt DFO}-{\tt GD} and {\tt SGD-GD} methods do not converge to a stationary point to ${\cal L}(\prm)$ even after more samples are observed. On the other hand, {\algoname} converges to a stationary point of ${\cal L}( \prm )$ at the rate of $ \| \grd {\cal L}( \prm ) \|^2 = {\cal O}(1/ S^{0.312} )$, matching Theorem~\ref{thm1} that predicts a rate of ${\cal O}(1/S^{1/3})$ (up to a logarithmic factor), where $S$ is the total number of samples observed.

Besides, with $\lambda = 0.5$, {\algoname} converges at a faster rate at the beginning (i.e., transient phase), but the fluctuation increases at the steady phase, as compared to a smaller $\lambda$ (e.g., $\lambda=0.25$). The estimator ${\bm g}_{\sf 2pt-I}$ is indeed biased, and converges to a non-stationary point to $\cal{L}(\prm)$. On the other hand, ${\bm g}_{\sf 2pt-II}$ converges to a stationary point, but uses twice as many samples as our one-point estimator for each update.

\textbf{Higher Dimension Case: Markovian Pricing.}
The second example examines a multi-dimensional ($d=5$) pricing problem similar to \citep[Sec.~5.2]{izzo2021learn}. The decision variable $\prm \in \RR^5$ denotes the prices of $d=5$ goods and $\kappa$ is a drifting parameter for the prices. Our goal is to maximize the \textit{Expected Revenue}, which is the opposite of performative risk $\EE_{Z\sim\Pi_{\prm} }[\ell(\prm;Z)]$ with $\ell(\prm;z)=-\pscal{\prm}{z}$, where $\Pi_{\prm} \equiv {\cal N}( {\bm \mu}_0-\kappa\prm,\sigma^2 {\bm I})$ is the unique stationary distribution of the Markov process
\vspace{-.1cm}
\[ \textstyle
Z_{t+1}=(1-\gamma) Z_{t} +\gamma \bar{Z}_{t+1}~\bar{Z}_{t+1} \sim {\cal N}( {\bm \mu}_0-\kappa \prm, \frac{2-\gamma}{\gamma}\sigma^2 {\bm I} ). \vspace{-.1cm}
\]
Note that in this case, the performative optimal solution is $\prm_{PO} = \argmin_{ \prm } {\cal L}( \prm ) = {\bm{\mu}_0}/{(2\kappa)}$.

We set $\gamma=0.1, \sigma=1$, drifting parameter $ \kappa=0.5,$ initial mean of non-shifted distribution ${\bm \mu}_0 = \left[5,-5,-5,5,-5\right]^\top$. All the algorithms are initialized by $\prm_0=\left[12,-12,12,-12,12\right]^\top$.
The performance of different algorithms in terms of expected revenue ($-\EE_{\Pi_{\prm}}\left[\ell(\prm; Z)\right]$) are contrasted in Figure \ref{fig:simulation} (middle), where {\tt OPT} denotes the optimal objective value.

Observe that {\algoname} algorithms converge to the highest expected reward, while the benchmarks {\tt SGD-GD} and {\tt DFO-GD} fail to find a solution with comparable performance. This is expected as the latter algorithms are at best guaranteed to converge to a performative stable solution. 

\textbf{Markovian Performative Regression.}
The last example considers linear regression problem in \citep{nagaraj2020least} which is a prototype problem for studying stochastic optimization with Markovian data. Such problems are rarely studied under both performativity and Markovianity.
Unlike the previous examples, this problem involves a pair of correlated r.v.s that follows a decision-dependent joint distribution. We adopt a setting similar to the regression example in \citep{izzo2021learn}, where $(X,Y) \sim \Pi_{\prm}$ with $X \sim {\cal N}(0, \sigma_1^2{\bm I}), Y|X \sim {\cal N} \left(\beta(\prm,X ), \sigma_2^2\right)$, $\beta(\prm, x)= \Pscal{x+\kappa\prm}{\prm_0}$. The loss function is the usual quadratic loss $\ell(\prm;x,y) = (\Pscal{x}{\prm}-y)^2$.
In this case, we define the \textit{Mean Squared Prediction Error} as the following performative risk:\vspace{-.1cm}\begin{align*}
 \textstyle 
{\cal L}(\prm) &=\EE_{\Pi_{\prm}}\left[\ell(\prm;X,Y)\right]=(\sigma_1^2+\kappa^2\norm{\prm_0}^2)\norm{\prm}^2 -2\sigma_1^2\Pscal{\prm}{\prm_0}+\sigma_1^2\norm{\prm_0}^2+\sigma_2^2,
\vspace{-.1cm}
\end{align*} 
In this experiment, we consider Markovian samples $(\Tilde{X}_{t},\Tilde{Y}_{t})_{t=1}^{T}$ drawn from the following AR process:\vspace{-.1cm}
\[
\begin{split}
& (\Tilde{X}_{t},\Tilde{Y}_{t}) = (1-\gamma)(\Tilde{X}_{t-1},\Tilde{Y}_{t-1})+ \gamma (X_{t},Y_{t}), \\
& X_{t} \sim {\cal N}(0, {\textstyle \frac{2-\gamma}{\gamma}} \sigma_1^2 \mathit{I}),~Y_{t}|X_t \sim {\cal N}(\beta(\prm_{t-1},X_{t}), {\textstyle \frac{2-\gamma}{\gamma}} \sigma_2^2),
\end{split}\vspace{-.1cm}
\] 
for any $t \geq 1$. We set $d=5$, $\prm_0=[5,-5,5,-5,5]^\top$, $\kappa=1/\norm{\prm_0}, \sigma_1^2=\sigma_2^2=1$, mixing parameter $\gamma=0.25$.
Figure \ref{fig:simulation} (right) shows the result of the simulation. 
Similar to the previous examples, we observe that {\tt DFO}-{\tt GD} and {\tt SGD} fail to find a stationary solution to ${\cal L}(\prm)$. Moreover, we include Miller's two-phase algorithm as in \cite{Miller2021OutsideTE}, where we change the minimization in the second phase to an SGD update, and it does not find the desired optimal objective value with $n=10^4$ (Markovian) sample gathered in the first phase. In contrast, {\algoname} converges to a near-optimal solution after a reasonable number of samples are observed.

\textbf{Conclusions.}
We have described a derivative-free optimization approach for finding a stationary point of the performative risk function. In particular, we consider a non-i.i.d.~data setting with samples generated from a controlled Markov chain and propose a two-timescale step sizes approach in constructing the gradient estimator. The proposed {\algoname} algorithm is shown to converge to a stationary point of the performative risk function at the rate of ${\cal O}(1/T^{1/3})$.

\bibliographystyle{plainnat}
\bibliography{reference.bib}

\newpage 
\appendix 
\section{Comparison to Related Works}\label{ass:lit}
This section provides a detailed comparison to related works on performative prediction under the stateful agent setting. This setting is relevant as the influences of the updated $\prm$ on the agent may not be manifested immediately due to the unforgetful nature of the agent. 
The recent works can be grouped into two categories in terms of the sought solution to \eqref{perf}: (i) finding the \emph{performative stable} solution satisfying $\prm_{PS} = \argmin_{ \prm \in \RR^d } \EE_{ Z \sim \Pi_{\prm_{PS}} } [ \ell( \prm; Z ) ]$, (ii) finding or approximating the \emph{performative optimal} solution that tackles \eqref{perf} directly. 

For seeking the \emph{performative stable} solution, 
\cite{brown2022performative} is the first to study population-based algorithms where the stateful agent updates the state-dependent distribution iteratively towards $\Pi_{\prm}$. The authors proved that under a special case when $k$ groups that form the mixture distribution for $\Pi_{\prm}$ respond slowly, then classical retraining algorithms converge to the performative stable solution. 

The follow-up works \citep{li2022state,roy2022projection} focus on more sophisticated stateful  agents and the reliance on past experiences of agents via controlled Markov Chain. In \citep{li2022state}, the authors developed gradient-type state-dependent stochastic approximation algorithm to achieve performative stable solution. In \citep{roy2022projection}, the authors proposed a stochastic conditional gradient-type algorithm with state-dependent Markovian data to tackle constrained nonconvex performative prediction problem. 

The search for (approximate) \emph{performative optimal} solution is challenging due to the non-convexity of \eqref{perf}. \citet{izzo2022learn} assumes that the transient distribution is parameterized by a low-dimensional vector and the distribution converges to $\Pi_{\prm}$ geometrically. Under these settings, the authors proposed to learn the distribution as a linear model to form an unbiased estimate of $\grd {\cal L}(\prm)$. The resultant algorithm follows a two-phases update approach: it first estimate the gradient correction term [cf.~second term in \eqref{eq:all-gradient}], followed by stochastic gradient update steps. 
Such approach has two main drawbacks: (i) estimating the gradient correction term requires strong prior assumptions on the distribution map (see e.g. Assumption 2 and 3 in \citep{izzo2022learn}), which limits its applicability, (ii) the estimation phase gathers a substantial amount of potentially sensitive information from data reaction patterns, which may incur privacy concern. Furthermore, it is noted that such procedure has a convergence rate of ${\cal O}(T^{-1/5})$ to stationary solution of ${\cal L}(\prm)$, which is outperformed by DFO in the current paper.

As mentioned in the main paper, adopting the DFO setting avoids the need to estimate the gradient correction term, which may necessitate additional assumptions on $\Pi_{\prm}$ as seen in \citep{izzo2022learn}. To this end, one of the first works to address performative optimal points with DFO method is \citep{ray2022decision} in the stateful agent setting. Notably, the analysis in \cite{ray2022decision} relies on (i) a mixture dominance assumption on $\Pi_{\prm}$, and (ii) a geometric decay environment assumption on the stateful agent.
In addition to relaxing the mixture dominance assumption, we remark that Assumption~\ref{assu:FastMixing} is relaxed from the geometric decay environment condition in \cite{ray2022decision}. For example, our setting covers general MDP models and the controlled AR(1) model, see \citep[Appendix A.1]{li2022state}.

\section{Proof of Lemma~\ref{lem:decomposition}} \label{ap:decompose}
\begin{proof}
Throughout this section, we let $\cprm_{k}\eqdef \prm_{k}+\delta_{k}u_{k}, g_{k}(\prm;u,z)\eqdef g_{\delta_{k}}(\prm;u,z)$ and ${\cal L}_{k}(\prm)\eqdef {\cal L}_{\delta_{k}}(\prm)$ for simplicity. We begin our analysis from Assumption \ref{assu:Lip} and the observation that $\prm_{k+1}-\prm_{k} = -\eta_{k} \sum_{m=1}^{\tau_{k}}\lambda^{\tau_{k}-m}g_{k}^{(m)}$. Recall that $g_{k}^{(m)}=\frac{d}{\delta_{k}}\ell\left(\check{\prm}_{k}^{(m)}; z_{k}^{(m)} \right)u_{k}$ and $\check{\prm}_{k}^{(m)} = \prm_{k}^{(m)} + \delta_{k} u_{k}$, we have
\[
    {\cal L}(\prm_{k+1}) - {\cal L}(\prm_{k}) + \eta_{k}\Pscal{\grd {\cal L}(\prm_{k})}{ \sum_{m=1}^{\tau_{k}}\lambda^{\tau_{k}-m}g_{k}^{(m)}} \leq \frac{L}{2}\eta_{k}^2\norm{\sum_{m=1}^{\tau_{k}}\lambda^{\tau_{k}-m}g_{k}^{(m)}}^2,
\]
Rearranging terms and adding $\frac{\eta_{k}}{1-\lambda}\norm{\grd {\cal L}(\prm_k)}^2$ on the both sides lead to
\begin{align*}
    \frac{\eta_{k}}{1-\lambda}\norm{\grd {\cal L}(\prm_k)}^2 &\leq {\cal L}(\prm_{k}) - {\cal L}(\prm_{k+1}) -\frac{\eta_{k}}{1-\lambda} \Pscal{\grd {\cal L}(\prm_{k})}{(1-\lambda)\sum_{m=1}^{\tau_{k}}\lambda^{\tau_{k}-m}g_{k}^{(m)}-\grd {\cal L}(\prm_k)}  + \frac{L}{2}\eta_{k}^2\norm{\sum_{m=1}^{\tau_{k}}\lambda^{\tau_{k}-m}g_{k}^{(m)}}^2
\end{align*}
Let ${\cal F}^k = \sigma( \prm_0, Z_{s}^{(m)}, u_s, 0 \leq s \leq k, 0 \leq m \leq \tau_{k})$ be the filtration of random variables. Taking expectation conditioned on ${\cal F}^{k-1}$
gives
\begin{align*}
    \frac{\eta_{k}}{1-\lambda}\norm{\grd {\cal L}(\prm_k)}^2 \leq &\EE_{{\cal F}^{k-1}}\left[{\cal L}(\prm_{k}) - {\cal L}(\prm_{k+1}) \right]
    \\
    &-\frac{\eta_{k}}{1-\lambda} \Pscal{\grd {\cal L}(\prm_k)}{(1-\lambda)\sum_{m=1}^{\tau_{k}}\lambda^{\tau_{k}-m}
    \EE_{{\cal F}^{k-1}}\left[g_{k}^{(m)}\right]
    -\grd {\cal L}(\prm_k)} 
    \\
    &+ \frac{L}{2}\eta_{k}^2\EE_{{\cal F}^{k-1}}\norm{\sum_{m=1}^{\tau_{k}}\lambda^{\tau_{k}-m}g_{k}^{(m)}}^2,
\end{align*}
By adding and subtracting, we obtain
\begin{align*}
    \frac{\eta_{k}}{1-\lambda}\norm{\grd {\cal L}(\prm_k)}^2 &\leq \EE_{{\cal F}^{k-1}}\left[{\cal L}(\prm_{k}) - {\cal L}(\prm_{k+1}) \right]
    \\
    &\quad -\frac{\eta_{k}}{1-\lambda} \Pscal{\grd {\cal L}(\prm_k)}{(1-\lambda)\sum_{m=1}^{\tau_{k}}\lambda^{\tau_{k}-m}\left(\EE_{{\cal F}^{k-1}}\left[g_{k}^{(m)}\right]-\EE_{Z\sim\Pi_{\cprm_k}, {\cal F}^{k-1}}\left[g_{k}(\prm_{k};u_k, Z)\right]\right)}
    \\
    &\quad -\frac{\eta_{k}}{1-\lambda} \Pscal{\grd {\cal L}(\prm_k)}{(1-\lambda)\sum_{m=1}^{\tau_{k}}\lambda^{\tau_{k}-m}\EE_{Z\sim\Pi_{\cprm_k},{\cal F}^{k-1}}\left[g_{k}(\prm_{k};u_k, Z)\right]-\grd {\cal L}(\prm_k)}
    \\
    &\quad + \frac{L}{2}\eta_{k}^2\EE_{{\cal F}^{k-1}}\norm{\sum_{m=1}^{\tau_{k}}\lambda^{\tau_{k}-m}g_{k}^{(m)}}^2
\end{align*}

By Lemma \ref{lem:BoundedBias}, the conditional expectation evaluates to $\EE_{Z\sim\Pi_{\cprm_k}}\left[g_{k}(\prm_{k};u_k, Z)\right]=\grd {\cal L}_{k}(\prm_{k})$. Dividing $\frac{\eta_{k}}{1-\lambda}$ derive that
\begin{align*}
    \norm{\grd {\cal L}(\prm_k)}^2 \leq &\frac{1-\lambda}{\eta_{k}} \EE_{{\cal F}^{k-1}}\left({\cal L}(\prm_{k}) - {\cal L}(\prm_{k+1}) \right)
    \\
    &-\Pscal{\grd {\cal L}(\prm_k)}{(1-\lambda)\sum_{m=1}^{\tau_{k}}\lambda^{\tau_{k}-m}\left(\EE_{{\cal F}^{k-1}}\left[g_{k}^{(m)}\right]-\EE_{{\cal F}^{k-1}}\EE_{Z\sim\Pi_{\cprm_k}}\left[g_{k}(\prm_{k};u_k,Z)|u_k\right]\right)}
    \\
    &-\Pscal{\grd {\cal L}(\prm_k)}{(1-\lambda)\left(\sum_{m=1}^{\tau_{k}}\lambda^{\tau_{k}-m}\grd {\cal L}_{k}(\prm_{k})\right)-\grd {\cal L}(\prm_k)}
    \\
    &+ \frac{L(1-\lambda)}{2}\eta_{k}\EE_{{\cal F}^{k-1}}\norm{\sum_{m=1}^{\tau_{k}}\lambda^{\tau_{k}-m}g_{k}^{(m)}}^2
\end{align*}
Summing over $k$ from 0 to $t$, indeed we obtain
\begin{align*}
    \sum_{k=0}^{t}\EE\norm{\grd {\cal L}(\prm_k)}^2 &\leq\sum_{k=0}^{t}\frac{1-\lambda}{\eta_{k}}\EE\left[{\cal L}(\prm_k)-{\cal L}(\prm_{k+1})\right]
    \\
    &\quad -\sum_{k=0}^{t}\EE\Pscal{\grd {\cal L}(\prm_k)}{(1-\lambda)\sum_{m=1}^{\tau_{k}}\lambda^{\tau_{k}-m}\left(\EE_{{\cal F}^{k-1}}\left[ g_{k}^{(m)}\right]-\EE_{{\cal F}^{k-1}}\EE_{Z\sim\Pi_{\cprm_k}}\left[g_{k}(\prm_{k};u_k,Z)|u_k\right]\right)}
    \\
    &\quad -\sum_{k=0}^{t}\EE\Pscal{\grd {\cal L}(\prm_k)}{(1-\lambda)\left(\sum_{m=1}^{\tau_{k}}\lambda^{\tau_{k}-m}\grd {{\cal L}_{k}(\prm_k)}\right)-\grd {\cal L}(\prm_k)}
    \\
    &\quad +\frac{L(1-\lambda)}{2}\sum_{k=0}^{t}\eta_{k}\EE \norm{\sum_{m=1}^{\tau_{k}}\lambda^{\tau_{k}-m}g_{k}^{(m)}}^2 \eqdef \term{1} + \term{2} + \term{3} + \term{4}
\end{align*}
\end{proof}

\section{Proof of Lemma~\ref{lem:bound_four_terms}}\label{App_bound_four_terms}
\begin{lemma}
\label{lem:bound_term_1}
Under Assumption \ref{assu:BoundLoss} and step size $\eta_{t}=\eta_{0}(1+t)^{-\alpha}$, it holds that
\begin{equation}
    {\bf I}_{1}(t)\leq c_1 (1-\lambda)(1+t)^{\alpha}
\end{equation}
\end{lemma}
where constant $c_1=\frac{2 \G}{\eta_0}$.
\begin{proof}

We observe the following chain
\begin{align*}
    \term{1} &= \sum_{k=0}^{t}\frac{1-\lambda}{\eta_{k}}\left(\EE\left[{\cal L}(\prm_k)\right]-\EE\left[{\cal L}(\prm_{k+1})\right]\right)
    \\
    &= (1-\lambda)\sum_{k=0}^{t} \EE[{\cal L}(\prm_{k})]/\eta_k-\EE[{\cal L}(\prm_{k+1})]/\eta_{k+1}+\EE[{\cal L}(\prm_{k+1})]/\eta_{k+1}-\EE[{\cal L}(\prm_{k+1})]/\eta_{k}
    \\
    &\overset{(a)}{=}(1-\lambda)\left[\EE[{\cal L}(\prm_0)/\eta_{0}]-\EE[{\cal L}(\prm_{t+1})/\eta_{t+1}]+\sum_{k=0}^{t}(\frac{1}{\eta_{k+1}}-\frac{1}{\eta_{k}})\EE[{\cal L}(\prm_{k+1})]\right]
    \\
    &\leq (1-\lambda)\max_{k} \left|\EE[{\cal L}(\prm_{k})]\right| \left(\frac{1}{\eta_{0}}+\frac{1}{\eta_{t+1}}+\frac{1}{\eta_{t+1}}-\frac{1}{\eta_{0}}\right)
\end{align*}
where equality (a) is obtained using the fact that step size $\eta_k>0$ is a decreasing sequence. Applying assumption \ref{assu:BoundLoss} to the last inequality leads to
\begin{align*}
    \term{1}
    &\leq (1-\lambda ) \G \frac{2}{\eta_{t+1}} \leq c_{1} (1-\lambda) (1+t)^{\alpha}
\end{align*}
where the constant $c_1=\frac{2 \G}{\eta_0}$.

\end{proof}

\begin{lemma}
\label{lem:bound_term_2}
Under Assumption
\ref{assu:Lip},
\ref{assu:BoundLoss},
\ref{assu:smooth_dist},
\ref{assu:FastMixing}, \ref{assu:smooth_kernel}, and constraint $0<2\alpha-4\beta<1$, and for all $k\geq 0 $, $\tau_{k}\geq\frac{1}{\log 1/\max\{\rho,\lambda\}}\log(1+k)$, then there exists universal constants $t_1,t_2>0$ such that
\begin{equation}
    {\bf I}_{2}(t) \leq c_2 \frac{d^2}{(1-\lambda)^2}{{\cal A}(t)}^{1/2}(1+t)^{1-(\alpha-2\beta)} \quad \forall t\geq \max\{t_1,t_2\}
\end{equation}
where ${\cal A}(t) \eqdef \frac{1}{1+t} \sum_{k=0}^{t} \EE\norm{\grd {\cal L}(\prm_k)}^2$ and $c_2 \eqdef \frac{\eta_0}{\delta_0^2}\frac{6\cdot (L_1\G^2+L_2\G^2+\sqrt{L}\G^{3/2} )}{\sqrt{1-2\alpha+4\beta}}$ is a constant.
\end{lemma}
\begin{proof}
Fix $k>0$, and recall $\cprm_{k}\eqdef\prm_{k}+\delta_{k}u_{k}, \cprm_{k}^{(\ell)}\eqdef\prm_{k}^{(\ell)}+\delta_{k}u_{k}$, then consider the following pair of Markov chains:
\begin{align}
    Z_{k} &=Z_{k}^{(0)}\xrightarrow[]{\cprm_{k}^{(1)}}Z_{k}^{(1)}\xrightarrow[]{\cprm_{k}^{(2)}}Z_{k}^{(2)}\xrightarrow[]{\cprm_{k}^{(3)}}Z_{k}^{(3)}\cdots\xrightarrow[]{\cprm_{k}^{(\tau_{k})}}Z_{k}^{(\tau_{k})}=Z_{k+1}
    \label{chain1}
    \\
    Z_k &=\Tilde{Z}_{k}^{(0)}\xrightarrow[]{\cprm_{k}}\Tilde{Z}_{k}^{(1)}\xrightarrow[]{\cprm_{k}}\Tilde{Z}_{k}^{(2)}\xrightarrow[]{\cprm_{k}}\Tilde{Z}_{k}^{(3)}\cdots\xrightarrow[]{\cprm_{k}}\Tilde{Z}_{k}^{(\tau_{k})}
    \label{chain2}
\end{align}
where the arrow associated with $\prm$ represents the transition kernel $\TT_{\prm}(\cdot,\cdot)$.

Note that Chain \ref{chain1} is the trajectory of {\sf DFO}($\lambda$) algorithm at iteration $k$, while 
Chain \ref{chain2} describes the trajectory of the same length generated by a reference Markov chain with fixed transition kernel $\TT_{\cprm_{k}}(\cdot,\cdot)$. 
Since $Z_{k}=Z_{k}^{(0)}=\Tilde{Z}_{k}^{(0)}$, we shall use them interchangeably.

\noindent
Define  $\Delta_{k, m}\eqdef \EE_{{\cal F}^{k-1}}\left[g_{k}^{(m)}-\EE_{Z\sim \Pi_{\cprm_k}}\left[g_{k} (\prm_{k};u_{k}, Z)\right]\right]$, then $\term{2}$ can be reformed as
\begin{align*}
    \term{2} &= -(1-\lambda)\EE\sum_{k=0}^{t}\Pscal{\grd {\cal L}(\prm_k)}{\sum_{m=1}^{\tau_{k}}\lambda^{\tau_{k}-m}\Delta_{k, m}}
    \\
    &\leq (1-\lambda) \EE\sum_{k=0}^{t} \norm{\grd {\cal L}(\prm_k)} \cdot \norm{ \sum_{m=1}^{\tau_{k}}\lambda^{\tau_{k} -m} \Delta_{k, m} }
\end{align*}
Next, observe that each $\Delta_{k, m}$ can be decomposed into 3 bias terms as follows
\begin{align*}
    \Delta_{k, m} &= \EE_{{\cal F}^{k-1}} \left[ \frac{d}{\delta_{k}}  \left(\EE[\ell(\cprm_{k}^{(m)}; Z_{k}^{(m)})|\cprm_{k}^{(m)}, Z_{k}^{(0)}] - \EE_{Z\sim \Pi_{\cprm_k}}[\ell(\cprm_k; Z)|\cprm_{k}] \right) u_{k}\right]
    \\
    &= \EE_{{\cal F}^{k-1}}\left[\frac{d}{\delta_{k}}  \left(\EE[\ell(\cprm_{k}^{(m)}; Z_{k}^{(m)})|\cprm_{k}^{(m)}, Z_{k}^{(0)}] - \EE_{\Tilde{Z}_{k}^{(m)}}[\ell(\cprm_{k}^{(m)}; \Tilde{Z}_{k}^{(m)})|\cprm_{k}^{(m)}, \Tilde{Z}_{k}^{(0)}] \right) u_{k}\right]
    \\
        &\quad + \EE_{{\cal F}^{k-1}} \left[ \frac{d}{\delta_{k}}   \left(\EE_{\Tilde{Z}_{k}^{(m)}}[\ell(\cprm_{k}^{(m)}; \Tilde{Z}_{k}^{(m)})|\cprm_{k}^{(m)}, \Tilde{Z}_{k}^{(0)}] - \EE_{Z\sim \Pi_{\cprm_k}}[\ell(\cprm_{k}^{(m)}; Z)|\cprm_{k}^{(m)}] \right) u_{k} \right]
        \\
        &\quad + \EE_{{\cal F}^{k-1}} \frac{d}{\delta_{k}} \underbrace{ \EE_{Z\sim\Pi_{\cprm_{k}}} \left[ \ell(\cprm_{k}^{(m)}; Z) - \ell(\cprm_k; Z)|\cprm_{k}^{(m)}, \cprm_{k}\right] }_{\leq 
        c_8
        \norm{\cprm_{k}^{(m)}-\cprm_{k}}+\frac{L}{2}\norm{\cprm_{k}^{(m)}-\cprm_{k}}^2} u_{k} 
\end{align*}
where we use Lemma \ref{lem:lip_decoupled_risk} in the last inequality and $c_8\eqdef 2\left(\sqrt{L\G} + \G L_1\right)$.

Here we bound these three parts separately. For the first term, it holds that
\begin{align*}
    &\left| \EE[\ell(\cprm_{k}^{(m)}; Z_{k}^{(m)})|\cprm_{k}^{(m)}, Z_{k}^{(0)}] - \EE_{\Tilde{Z}_{k}^{(m)}}[\ell(\cprm_{k}^{(m)}; \Tilde{Z}_{k}^{(m)})|\cprm_{k}^{(m)}, \Tilde{Z}_{k}^{(0)}] \right|
    \\
    =&\left| \int_{\sf Z}\ell(\cprm_{k}^{(m)}; z)\PP(Z_{k}^{(m)}=z|Z_{k}^{(0)})-\ell(\cprm_{k}^{(m)}; z)\PP(\Tilde{Z}_{k}^{(m)}=z|\Tilde{Z}_{k}^{(0)})dz
    \right|
    \\
    \leq &G\int_{\sf Z}\left|\PP(Z_{k}^{(m)}=z|Z_{k}^{(0)})- \PP(\Tilde{Z}_{k}^{(m)}=z|\Tilde{Z}_{k}^{(0)})\right| dz
    \\
    = &2 \G \tv{\PP(z_{k}^{(m)}\in\cdot|Z_{k}^{(0)})}{\PP(\Tilde{Z}_{k}^{(m)}\in\cdot|Z_{k}^{(0)})}
    \\
    \leq & 2 \G L_2 \sum_{\ell=1}^{m-1} \norm{\cprm_{k}^{(\ell)} -\cprm_k }
    = 2 \G L_2 \sum_{\ell=1}^{m-1} \norm{\prm_{k}^{(\ell)} -\prm_k }
\end{align*}
where the first inequality is due to Assumption \ref{assu:BoundLoss}, the second inequality is due to Lemma \ref{lem:tv_summation_bound}.

For the second term,
we have
\begin{align*}
    \left|\EE_{\Tilde{Z}_{k}^{(m)}}[\ell(\cprm_{k}^{(m)}; Z_{k}^{(m)})] - \EE_{Z\sim \Pi_{\cprm_k}}[\ell(\cprm_k; Z)]\right| &=\left|\int_{\sf Z}\ell(\cprm_{k}^{(m)}; z)\PP(\Tilde{Z}_{k}^{(m)}=z|\Tilde{Z}_{k}^{(0)})-\ell(\cprm_{k}^{(m)}; z)\Pi_{\cprm_k}(z) d z\right|
    \\
    &\overset{(a)}{\leq} \G\int_{\sf Z}|\PP(\Tilde{Z}_{k}^{(m)}=z|\Tilde{Z}_{k}^{(0)})- \Pi_{\cprm_k}(z))| dz
    \\
    & = 2 \G \tv{\PP(\Tilde{Z}_{k}^{(m)}\in\cdot|\Tilde{Z}_{k}^{(0)})}{\Pi_{\cprm_k}}
    \\
    &\overset{(b)}{\leq} 2\G M\rho^{m}
\end{align*}
where we use Assumption \ref{assu:BoundLoss} in inequality (a) and Assumptions \ref{assu:FastMixing} in inequality (b). 
Combining three upper bounds, we obtain that
\begin{align*}
   \norm{\Delta_{k,m}}&\leq \EE_{{\cal F}^{k-1}}\frac{d}{\delta_{k}}\left(2 \G L_2 \sum_{\ell=1}^{m-1}\left[ \norm{\prm_{k}^{(\ell)} -\prm_k }\right] + 2\G M\rho^{m} + c_8 \norm{\cprm_{k}^{(m)}-\cprm_k}+\frac{L}{2} \norm{\cprm_{k}^{(m)}-\cprm_k}^2\right)
    \\
    &\leq\frac{d}{\delta_{k}}\left(2 L_2 \G \sum_{\ell=1}^{m-1} \sum_{j=1}^{\ell-1}\eta_{k}\lambda^{\tau_{k}-j}\frac{d \G}{\delta_{k}} + 2\G M\rho^{m} + c_8 \sum_{j=1}^{m-1}\eta_{k}\lambda^{\tau_{k}-j}\frac{d \G}{\delta_{k}}\right)
    \\
    &+\frac{d}{\delta_{k}}\frac{L}{2}
    \left(\sum_{j=1}^{m-1}\eta_{k}\lambda^{\tau_{k}-j}\frac{d \G}{\delta_{k}}\right)^2
    \\
    &<\frac{d}{(1-\lambda)^2}\left(2L_2\G^2 d+c_8 \G d\right)\lambda^{\tau_{k}-m+1}\frac{\eta_{k}}{\delta_{k}^2}+\frac{L\G^2 d^3}{2(1-\lambda)^2}\lambda^{2(\tau_{k}-m+1)}\frac{\eta_{k}^2}{\delta_{k}^3}+2\G M d\frac{\rho^m}{\delta_{k}}
\end{align*}
Then it holds that
\begin{align*}
    \norm{\sum_{m=1}^{\tau_{k}}\lambda^{\tau_{k}-m}\Delta_{k,m}} &\leq \sum_{m=1}^{\tau_{k}}\lambda^{\tau_{k}-m}\norm{\Delta_{k,m}}
    \\
    &\leq \frac{d}{(1-\lambda)^2}\left(2L_2\G^2d+c_8d\G\right)\frac{\eta_{k}}{\delta_{k}^2}\sum_{m=1}^{\tau_{k}}\lambda^{2(\tau_{k}-m)}\lambda
    \\
    &\quad +\frac{L\G^2}{2(1-\lambda)^2}d^3\frac{\eta_{k}^2}{\delta_{k}^3}\sum_{m=1}^{\tau_{k}}\lambda^{3(\tau_{k}-m)}\lambda^2
    +2\G M d\delta_{k}^{-1}\sum_{m=1}^{\tau_{k}}\lambda^{\tau_{k}-m}\rho^{m}
    \\
    &\leq (2L_2\G^2 d+c_8 G d)\frac{d\lambda}{(1-\lambda)^3}\frac{\eta_{k}}{\delta_{k}^2} +\frac{L\G^2}{2}\frac{d^3\lambda^2}{1-\lambda}\frac{\eta_{k}^2}{\delta_{k}^3}
    +2\G M d\delta_{k}^{-1}\tau_{k}\max\{\rho,\lambda\}^{\tau_{k}}
\end{align*}
Finally, provided $\tau_{k}\geq\log_{\max\{\rho,\lambda\}}(1+k)^{-1}$ and $0<2\alpha-4\beta<1$, we can bound $\term{2}$ as follows:
\begin{align*}
    \term{2} &\leq (1-\lambda) \EE\sum_{k=0}^{t} \norm{\grd {\cal L}(\prm_k)} \cdot \norm{\sum_{m=1}^{\tau_{k}}\lambda^{\tau_{k} -m} \Delta_{k,m}}
    \\
    &\leq (1-\lambda) \EE\sum_{k=0}^{t} \norm{\grd {\cal L}(\prm_k)}\left[(2 L_2\G^2 d+c_8 G d)\frac{d\lambda}{(1-\lambda)^3}\frac{\eta_{k}}{\delta_{k}^2}+\frac{L\G^2}{2}\frac{d^3\lambda^2}{1-\lambda}\frac{\eta_{k}^2}{\delta_{k}^3}
    +2\G M d\frac{\tau_{k}}{\delta_{k}(1+k)}\right]
    \\
    &\leq \frac{d\lambda}{(1-\lambda)^2}(2L_2\G^2 d+c_8 \G d)\left(\sum_{k=0}^{t}\EE\norm{\grd {\cal L}(\prm_k)}^2\right)^{1/2}\left(\sum_{k=0}^{t}\frac{\eta_{k}^2}{\delta_{k}^4}\right)^{1/2}
    \\
    &\quad +d^3\lambda^2 \frac{L\G^2}{2}\left(\sum_{k=0}^{t}\EE\norm{\grd {\cal L}(\prm_k)}^2\right)^{1/2}\left(\sum_{k=0}^t\frac{\eta_{k}^4}{\delta_{k}^6}\right)^{1/2}
    \\
    &\quad +\frac{d}{1-\lambda}\G M \left(\sum_{k=0}^{t}\EE\norm{\grd {\cal L}(\prm_k)}^2\right)^{1/2}\left(\sum_{k=0}^{t}\frac{\tau_{k}^2}{\delta_{k}^2(1+k)^2}\right)^{1/2}
    \\
    &\overset{(b)}{\leq} \frac{d^{2}\lambda}{(1-\lambda)^2}6(L_2\G^2+\sqrt{L}\G^{3/2}+L_1\G^2 )\left(\sum_{k=0}^{t}\EE\norm{\grd {\cal L}(\prm_k)}^2\right)^{1/2}\left(\sum_{k=0}^{t}\frac{\eta_{k}^2}{\delta_{k}^4}\right)^{1/2}
    \\
    &\leq c_{2}\frac{d^{2}}{(1-\lambda)^2}\left(\frac{1}{1+t}\sum_{k=0}^{t}\EE\norm{\grd {\cal L}(\prm_k)}^2\right)^{1/2} \cdot (1+t)^{1-(\alpha-2\beta)}\quad \forall t\geq \max\{t_1, t_2\}
\end{align*}
\noindent
where $c_2 \eqdef \frac{\eta_0}{\delta_0^2}\frac{6\cdot (L_1\G^2+L_2\G^2+\sqrt{L}\G^{3/2} )}{\sqrt{1-2\alpha+4\beta}}$. The inequality (b) holds since $\tau_{k}=\Theta(\log k)$, $4\alpha-6\beta>2\alpha-4\beta$ and $2-2\beta>2\alpha-4\beta$, so there exist constants 
\begin{align}
    t_{1} & \eqdef \inf_{t}\left\{ t\geq 0 \, | \, \frac{d^6\lambda^4 L^2\G^4}{4}\sum_{k=0}^{t} \frac{\eta_{k}^4}{\delta_{k}^6}\leq \frac{d^2\lambda^2(2L_2\G^2d+c_8\G d)^2}{(1-\lambda)^4}\sum_{k=0}^{t} \frac{\eta_{k}^{2}}{\delta_{k}^{4}}\right\} \label{eq:const-t1}
    \\
    t_{2} & \eqdef \inf_{t}\left\{ t\geq 0 \, | \, d^2 \G^2M^2\sum_{k=0}^{t} \frac{\tau_{k}^2}{\delta_{k}^2 (1+k)^2}\leq \frac{d^2\lambda^2(2L_2\G^2d+c_8\G d)^2}{(1-\lambda)^4}\sum_{k=0}^{t} \frac{\eta_{k}^{2}}{\delta_{k}^{4}}\right\} \label{eq:const-t2}
\end{align}
In brief, we have
\begin{align*}
    \term{2} \leq c_{2}\frac{d^{2}}{(1-\lambda)^2}\left(\frac{1}{1+t}\sum_{k=0}^{t}\EE\norm{\grd {\cal L}(\prm_k)}^2\right)^{1/2} \cdot (1+t)^{1-(\alpha-2\beta)}\quad \forall t\geq \max\{t_1,t_2\}
\end{align*}
\end{proof}

\begin{lemma}
\label{lem:bound_term_3}
Under Assumption \ref{assu:Lip}, \ref{assu:BoundLoss} and $0<\beta< 1/2$, with $\tau_{k}\geq\frac{1}{\log 1/\max\{\rho,\lambda\}}\left(\log(1+k)+\log\frac{d}{\delta_0}\right)$, it holds that
\begin{equation}
    {\bf I}_{3}(t)\leq c_3 {{\cal A}(t)}^{\frac{1}{2}} (1+t)^{1-\beta}
\end{equation}
where ${\cal A}(t) \eqdef \frac{1}{1+t} \sum_{k=0}^{t} \EE\norm{\grd {\cal L}(\prm_k)}^2$ and constant $c_3=\frac{1}{\sqrt{1-2\beta}}\max\{2^{1-\beta}L\delta_0, 2^\beta \G \sqrt{1-\beta}\}$.
\end{lemma}
\begin{proof}
Recall that  $g_{k}(\prm;u,z)\eqdef g_{\delta_{k}}(\prm;u,z)$ and ${\cal L}_{k}(\prm)\eqdef {\cal L}_{\delta_{k}}(\prm)$.
\begin{align*}
    \term{3} &= -\sum_{k=0}^{t} \EE \Pscal{\grd {\cal L}(\prm_k)}{(1-\lambda)\left(\sum_{m=1}^{\tau_{k}}\lambda^{\tau_{k}-m} \grd{{\cal L}_{k}(\prm_{k})} \right) -\grd{\cal L}(\prm_k)}
    \\
    &= -\sum_{k=0}^{t} \EE \Pscal{\grd {\cal L}(\prm_k)}{\left((1-\lambda)\sum_{m=1}^{\tau_{k}}\lambda^{\tau_{k}-m}\right) \grd{{\cal L}_{k}(\prm_{k})}  -\grd {\cal L}(\prm_k)}
    \\
    &= -\sum_{k=0}^{t} \EE \Pscal{\grd {\cal L}(\prm_k)}{\grd{{\cal L}_{k}(\prm_{k})} -\grd {\cal L}(\prm_k)} - \lambda^{\tau_{k}}\EE\Pscal{\grd {\cal L}(\prm_k)}{\EE_{Z\sim \Pi_{\cprm_k}}[g_{k}(\prm_{k};u_{k},Z)]}
\end{align*}
where we apply Lemma \ref{lem:UnbiasedGrad} at the last equality. 

By triangle inequality, Cauchy-Schwarz inequality and Assumption \ref{assu:BoundLoss}, we obtain
\begin{align*}
    \term{3} &\leq \sum_{k=0}^{t} \EE \norm{\grd {\cal L}(\prm_k)} \cdot \norm{\grd{{\cal L}_{k}(\prm_{k})} - \grd {\cal L}(\prm_k)}
    + \sum_{k=0}^{t} \lambda^{\tau_{k}} \EE\norm{\grd {\cal L}(\prm_k)} \frac{d\G}{\delta_{k}}
\end{align*}

Provided $\tau_{k}\geq\frac{\log(1+k)+\log\frac{d}{\delta_0}}{\log 1/\max\{\rho,\lambda\}} \geq \frac{ \log{\delta_0 / d(1+k)^{-1}}}{\log \max\{\rho,\lambda\}} = \log_{\max\{\rho,\lambda\}}\frac{\delta_0}{d} (1+k)^{-1}\geq \log_{\lambda} \frac{\delta_0}{d}(1+k)^{-1}$, with Lemma \ref{lem:BoundedBias} as a consequence of Assumption \ref{assu:Lip}, we have

\begin{align*}
    \term{3} &\leq  \sum_{k=0}^{t} \EE \norm{\grd {\cal L}(\prm_k)} \cdot L\delta_{k} + \sum_{k=0}^{t}  \EE\norm{\grd {\cal L}(\prm_k)} \frac{\delta_0}{d} \frac{d\G}{\delta_{0}} (1+k)^{\beta-1}
    \\
    &=  \sum_{k=0}^{t} \EE \norm{\grd {\cal L}(\prm_k)} \cdot L\delta_{k} + \G \sum_{k=0}^{t}  \EE\norm{\grd {\cal L}(\prm_k)} (1+k)^{\beta-1}
    \\
    &\leq L\left( \sum_{k=0}^{t}\EE\norm{\grd {\cal L}(\prm_k)}^2\right)^{1/2} \left(\sum_{k=0}^{t}\delta_{k}^2 \right)^{1/2}
    + \G\left(\sum_{k=0}^{t} \EE\norm{\grd {\cal L}(\prm_k)}^2\right)^{1/2} \left( \sum_{k=0}^{t} (1+k)^{2(\beta-1)}\right)^{1/2}
\end{align*}

\noindent
Since $\beta<1/2$, it holds that
\begin{align*}
    &\sum_{k=0}^{t} \delta_{k}^2 = \sum_{k=0}^{t} \frac{\delta_{0}^2}{(1+k)^{2\beta}} \leq \frac{\delta_{0}^{2}}{1-2\beta} \left[ 1-2\beta+(1+t)^{1-2\beta}-1\right] \leq \frac{\delta_{0}^{2}}{1-2\beta}(1+t)^{1-2\beta}
    \\
    &\sum_{k=0}^{t} (1+k)^{2(\beta-1)} 
    \leq 1+\int_{0}^{t} (x+1)^{2(\beta-1)} \diff x < 1+ \frac{1}{1-2\beta}
\end{align*}
Then we can conclude
\begin{align*}
    \term{3} &\leq  c_{3} \left( \frac{1}{1+t}\sum_{k=0}^{t}\EE\norm{\grd {\cal L}(\prm_k)}^2\right)^{1/2} \cdot (1+t)^{1-\beta}
\end{align*}
where $c_{3} \eqdef \frac{2}{\sqrt{1-2\beta}}\max\{L\delta_0, \G \sqrt{1-\beta}\}$.
\end{proof}

\begin{lemma}
\label{lem:bound_term_4}
Under assumption \ref{assu:BoundLoss} and constraint $0<\alpha<1$, it holds that 
\begin{equation}
    \term{4}\leq c_4 \frac{d^2}{1-\lambda} (1+t)^{1-\left(\alpha-2\beta\right)}
\end{equation}
where constant $c_4=\frac{\eta_{0} L \G^2 }{\delta_{0}^2 (2\beta-\alpha+1)}$.
\end{lemma}

\begin{proof}
\begin{align*}
    \term{4} &= \frac{(1-\lambda)L}{2}\sum_{k=0}^{t} \eta_{k} \EE \norm{\sum_{m=1}^{\tau_{k}}\lambda^{\tau_{k}-m}g_{k}^{(m)}}^2
    \\
    &\leq\frac{(1-\lambda)L}{2}\sum_{k=0}^{t} \eta_{k}\EE\left(\sum_{m=1}^{\tau_{k}}\lambda^{\tau_{k}-m}\norm{{g_{k}^{(m)}}}\right)^2
    \\
    &\leq\frac{(1-\lambda)L}{2}\sum_{k=0}^{t} \eta_{k}\left(\sum_{m=1}^{\tau_{k}}\lambda^{\tau_{k}-m}\right)^2\frac{(d\G)^2}{\delta_{k}^2}
    \\
    &\leq\frac{(1-\lambda)L d^2\G^2}{2}\sum_{k=0}^{t}\left(\frac{1-\lambda^{\tau_{k}}}{1-\lambda}\right)^2\frac{\eta_{k}}{\delta_{k}^2}
    \\
    &<\frac{d^2 L\G^2}{2(1-\lambda)}\sum_{k=0}^{t}\frac{\eta_{k}}{\delta_{k}^2}
\end{align*}

Recall that $\eta_{k}=\frac{\eta_{0}}{(k+1)^{\alpha}}$, $\delta_{k}=\frac{\delta_{0}}{(1+k)^{\beta}}$ and $\alpha<1, \beta\geq 0$, it is clear that $\alpha-2\beta<1$, so it holds that
\begin{align*}
    \sum_{k=0}^{t}\frac{\eta_{k}}{\delta_{k}^2} &= \frac{\eta_{0}}{\delta_{0}^{2}} \sum_{k=0}^{t} (1+k)^{2\beta-\alpha} \leq \frac{\eta_{0}}{\delta_{0}^{2}} \left(1+\int_{0}^{t} (1+x)^{2\beta-\alpha} \diff x\right)
    \\
    &\leq \frac{\eta_{0}}{\delta_{0}^2 (2\beta-\alpha+1)} \left[ (1+t)^{2\beta-\alpha+1} -\alpha+2\beta \right]
    \leq \frac{2\eta_{0}}{\delta_{0}^2 (2\beta-\alpha+1)} (1+t)^{2\beta-\alpha+1}
\end{align*}
In conclusion, we obtain that
\begin{align*}
    \term{4} \leq d^2\frac{L\G^2}{1-\lambda}  \frac{\eta_{0}}{\delta_{0}^2 (2\beta-\alpha+1)} \cdot (1+t)^{2\beta-\alpha+1} = c_{4}\frac{d^2}{1-\lambda} (1+t)^{1-(\alpha-2\beta)}
\end{align*}
where $c_{4} \eqdef \frac{\eta_{0} }{\delta_{0}^2 }\cdot\frac{L \G^2}{2\beta-\alpha+1}.$ 
\end{proof}

\section{Proof of Lemma \ref{lem:major_bound}}
\begin{proof}
    Combining Lemmas \ref{lem:decomposition} and  \ref{lem:bound_four_terms}, subject to the constraints $0<\alpha<1,0<\beta\leq 1/2, 0<2\alpha-4\beta\leq 1$, it holds that for any $t\geq\max\{t_1,t_2\}$,
\begin{align*}
    &\sum_{k=0}^{t}\EE\norm{\grd {\cal L}(\prm_k)}^2 
    \\
    &\leq\term{1} + \term{2} + \term{3} + \term{4}
    \\
    & \leq c_1 (1-\lambda) (1+t)^{\alpha} + c_2 \frac{d^{5/2}}{(1-\lambda)^2} (1+t)^{1-(\alpha-2\beta)} {\cal A}(t)^{1/2}
    \\
    &\quad +c_3 (1+t)^{1-\beta}{\cal A}(t)^{1/2}  +c_4 \frac{d^2}{1-\lambda} (1+t)^{1-(\alpha-2\beta)}
\end{align*}
Recall ${\cal A}(t) \eqdef \frac{1}{1+t} \sum_{k=0}^{t}\EE\norm{\grd {\cal L}(\prm_k)}^2$, above inequality can be rewritten as
\begin{align*}
    {\cal A}(t) &\leq \frac{1}{1+t}\bigg[c_2 \frac{d^{5/2}}{(1-\lambda)^2} (1+t)^{1-(\alpha-2\beta)} {\cal A}(t)^{1/2}
    \\
    &\quad +c_3 (1+t)^{1-\beta}{\cal A}(t)^{1/2} 
    +c_1 (1-\lambda) (1+t)^{\alpha}+c_4 \frac{d^2}{1-\lambda}(1+t)^{1-(\alpha-2\beta)}\bigg]
    \\
    &= \left(c_2 \frac{d^{5/2}}{(1-\lambda)^2} (1+t)^{-(\alpha-2\beta)}+c_3 (1+t)^{-\beta}\right){\cal A}(t)^{1/2} 
    + c_1 (1-\lambda) (1+t)^{-(1-\alpha)}
    \\
    &\quad + c_4 \frac{d^2}{1-\lambda} (1+t)^{-(\alpha-2\beta)}
\end{align*}
which is a quadratic inequality in ${\cal A}(t)^{1/2}$.

Let $x={\cal A}(t)^{1/2}, a=c_2 \frac{d^{5/2}}{(1-\lambda)^2} (1+t)^{-(\alpha-2\beta)}+c_3 t^{-\beta}, b=c_1 (1-\lambda) (1+t)^{-(1-\alpha)}+ c_4 \frac{d^2}{1-\lambda} (1+t)^{-(\alpha-2\beta)}$, we have $x^2-ax-b\leq 0$. Since $a,b>0$, the quadratic has two real roots, denoted as $x_1,x_2$ respectively, and $x_1<0<x_2$. Moreover, we must have $x\leq x_2$, which implies $x\leq\frac{a+\sqrt{a^2+4b}}{2}\leq \frac{a+a+2\sqrt{b}}{2}=a+\sqrt{b}$.
Therefore, ${\cal A}(t)=x^2\leq (a+\sqrt{b})^2\leq 2(a^2+b)$.
Substituting $a,b$ back leads to
\begin{align*}
    {\cal A}(t) 
    &\leq 2\left(c_2 \frac{d^{5/2}}{(1-\lambda)^2} (1+t)^{-(\alpha-2\beta)}+c_3 (1+t)^{-\beta}\right)^2
    +2c_1 (1-\lambda) (1+t)^{-(1-\alpha)}
        \\
        &\quad + 2c_4 \frac{d^2}{1-\lambda} (1+t)^{-(\alpha-2\beta)}
    \\
    &\overset{(a)}{\leq} 4 c_2^2 \frac{d^5}{(1-\lambda)^4} (1+t)^{-2(\alpha-2\beta)}+ 4 c_3^2 (1+t)^{-2\beta}
    + 2 c_1 (1-\lambda) (1+t)^{-(1-\alpha)}
        \\
        &\quad + 2 c_4 \frac{d^2}{1-\lambda} (1+t)^{-(\alpha-2\beta)}
    \\
    &\leq 4 c_3^2 (1+t)^{-2\beta} + 2 c_1 (1-\lambda) (1+t)^{-(1-\alpha)}+ 4 c_4 \frac{d^2}{1-\lambda} (1+t)^{-(\alpha-2\beta)},
\end{align*}
\noindent
where inequality (a) is due to the fact $(x+y)^2\leq 2 (x^2+y^2)$, the last inequality holds because there exists sufficiently large constant $t_3$ such that, $4 c_2^2 \frac{d^5}{(1-\lambda)^4} (1+t)^{-2(\alpha-2\beta)}\leq 2 c_4 \frac{d^2}{1-\lambda} (1+t)^{-(\alpha-2\beta)} \forall t\geq t_3$ given $\alpha>2\beta$. 
Therefore, set $t_0\eqdef \max\{t_1,t_2,t_3\}$, then for all $t\geq t_0$, we have
\begin{align*}
    {\cal A}(t) &\leq 4 \max\{ c_1 (1-\lambda), c_3^2 , c_4 \frac{d^2}{1-\lambda}\} 
    \cdot \left((1+t)^{-2\beta}+(1+t)^{-(1-\alpha)}+(1+t)^{-(\alpha-2\beta)}\right)
    \\
    &\leq 12 \max\{ c_1 (1-\lambda), c_3^2 , c_4 \frac{d^2}{1-\lambda}\} (1+t)^{-\min\{2\beta,1-\alpha,\alpha-2\beta\}}
\end{align*}
Recall that constant $c_1$ contains $1/\eta_0$, $c_3$ contains $\delta_0$, $c_4$ contains $\eta_0/\delta_0^2$, , thus we can set $\delta_0 = d^{1/3}, \eta_0 = d^{-2/3}$, which yields
\begin{align*}
    {\cal A}(t) &\leq 12 \max\{c_5 (1-\lambda), c_6, \frac{c_7}{1-\lambda}\} d^{2/3} (1+t)^{-\min\{2\beta,1-\alpha,\alpha-2\beta\}}
\end{align*}
where constants 
\begin{align*}
   c_5 =2 \G,\quad 
   c_6 =\frac{4 \max\{L^2, \G^2(1-\beta)\}}{1-2\beta}, \quad
   c_7 =\frac{L\G^2 }{2\beta-\alpha+1}
\end{align*}
do not contain $\eta_0$ and $\delta_0$.
Moreover, note that $\max_{\alpha,\beta}\min\{2\beta,1-\alpha,\alpha-2\beta\}=\frac{1}{3}$, thus it holds
\begin{align*}
    &\frac{1}{1+T}\sum_{k=0}^{T}\EE\norm{\grd {\cal L}(\prm)_k}^2 
    \leq 12 \max\{c_5 (1-\lambda), c_6, \frac{c_7}{1-\lambda}\} d^{2/3} (1+T)^{-1/3}
\end{align*}
where the rate ${\cal O}(1/T^{1/3})$ can be attained by choosing $\alpha=\frac{2}{3}$, $\beta=\frac{1}{6}.$ This immediately leads to Theorem \ref{thm1} by observing 
\[
    \min_{0\leq k \leq T} \EE\norm{\grd {\cal L}(\prm_k)}^2\leq \textstyle \frac{1}{1+T}\sum_{k=0}^{T}\EE\norm{\grd {\cal L}(\prm_k)}^2.
\]
\end{proof}

\section{Non-smooth Analysis}
In this section, we aim to apply our algorithm to non-smooth performative risk optimization problem and analyze its convergence rate. Before presenting the theorem, we need the following Lipschitz loss assumption \ref{assu:lip_ell}.
\begin{assumption}{\bf (Lipschitz Loss)}\label{assu:lip_ell}
There exists constant $L_0>0$ such that 
\[
    \left|\ell(\prm_{1};z)-\ell(\prm_{2};z)\right|\leq L_{0}\norm{\prm_{1}-\prm_{2}}, ~\forall~ \prm_{1},\prm_{2}\in \RR^d, ~\forall ~z\in {\sf Z}
\]
\end{assumption}

Under Assumption \ref{assu:lip_ell} and some other regularity conditions, one can show that the performative risk is also Lipschitz continuous. Formally, this can be stated as follows.
\begin{lemma}\label{lem:composite_lip}
Under Assumption \ref{assu:lip_ell}, \ref{assu:BoundLoss}, \ref{assu:smooth_dist}, the performative risk ${\cal L}(\prm)$ is ($L_0+2 L_1\G $)-Lipschitz continuous.
\end{lemma}
Under non-smooth settings, the convergence behavior can be characterized in both squared gradient norm and proximity gap. Now, we are ready to show the following theorem:
\begin{theorem}\label{thm:non-smooth}
{\bf (\algoname~for Non-smooth Optimization)}
Under Assumption
\ref{assu:lip_ell},
\ref{assu:BoundLoss}, 
\ref{assu:smooth_dist}, \ref{assu:FastMixing}, \ref{assu:smooth_kernel}, with two time-scale step sizes $\eta_k=\eta_0 (1+k)^{-\alpha}, \delta_k=d (1+k)^{-\beta}, \tau_{k}\geq\frac{\log(1+k)}{\log 1/\max\{\rho,\lambda\}}$, where $\alpha,\beta$ satisfies $0<3\beta<\alpha<1$, there exists a constant $t_4$ such that, the iterates $\{\prm_k\}_{k\geq 1}$ satisfies for all $T\geq t_4$
\[
    \frac{1}{1+T}\sum_{k=0}^{T}\EE\norm{\grd{{\cal L}_{\delta_{k}}(\prm_k)}}^2 
    ={\cal O}(T^{-\min\{1-\alpha, \alpha-3\beta\}})
\]
and the following error estimate holds for all $T > 0$ and $\prm\in\mathbb{R}^{d}$
\[
\frac{1}{1+T}\sum_{k=0}^{T}\EE|{\cal L}_{\delta_{k}}(\prm)-{\cal L}(\prm)|={\cal O}(T^{-\beta})
\]
\end{theorem}
\begin{Corollary}
\label{cor:non-smooth}
($\epsilon$-stationarity, $\mu$-proximity)
Suppose Assumptions of Theorem \ref{thm:non-smooth} hold. Fix any $\epsilon,\mu>0$, for $T= \max\{{\cal O}(1/\epsilon^4), {\cal O}(1/\mu^6)\}$, the following estimates hold simultaneously
\[\frac{1}{1+T} \sum_{k=0}^{T}\EE\norm{\grd {\cal L}_{\delta_{k}}(\prm_k)}^2 \leq \epsilon\]
\[\frac{1}{1+T} \sum_{k=0}^{T}\EE\left|{\cal L}_{\delta_{k}}(\prm_{k})-{\cal L}(\prm_{k})\right|\leq \mu\]
\end{Corollary}

Next, we present the proof of Theorem \ref{thm:non-smooth}.
\begin{proof}
This proof resembles the proof of Lemma \ref{lem:bound_four_terms}, where we reinterpret $\sum_{k=0}^{t}\EE\norm{\grd {\cal L}(\prm_{k})}^2$ as $\sum_{k=0}^{t}\EE\norm{\grd {\cal L}_{\delta_{k}}(\prm_{k})}^2$, and ${\cal L}(\prm_{k})$ as ${\cal L}_{\delta_{k}}(\prm_{k})$, with additional bias terms that, as we shall prove, are not dominant.

Due to Lemma \ref{lem:composite_lip}, $\cal{L}(\prm)$ is $(L_0+2 L_1\G) $-Lipschitz. Then by Lemma \ref{lem:UnbiasedGrad}, ${\cal L}_{\delta}(\prm)$ is $\frac{d}{\delta}(L_0 + 2 L_1 \G)$-smooth for all $\delta>0$. Similar to Lemma \ref{lem:decomposition}, we have
\begin{align*}
    &{\cal L}_{\delta_{k}}(\prm_{k+1})-{\cal L}_{\delta_{k}}(\prm_{k})+\frac{\eta_{k}}{1-\lambda}\Pscal{\grd{\cal L}_{\delta_{k}}(\prm_{k})}{(1-\lambda)\sum_{m=1}^{\tau_{k}}\lambda^{\tau_{k}-m}g_{k}^{(m)}} 
    \\
    &\leq\frac{d (L_0+2 L_1\G) }{2\delta_{k}}\eta_{k}^2\norm{\sum_{m=1}^{\tau_{k}}\lambda^{\tau_{k}-m}g_{k}^{(m)}}^2
\end{align*}

By adding, subtracting and rearranging terms, after taking conditional expectation on ${\cal F}^{k-1}$, it holds that
\begin{align*}
    \frac{\eta_{k}}{1-\lambda}\norm{\grd {\cal L}_{\delta_{k}}(\prm_k)}^2 &\leq \EE_{{\cal F}^{k-1}}\left[{\cal L}_{\delta_{k}}(\prm_{k}) - {\cal L}_{\delta_{k+1}}(\prm_{k+1}) + {\cal L}_{\delta_{k+1}}(\prm_{k+1}) - {\cal L}_{\delta_{k}}(\prm_{k+1}) \right]
    \\
    &\quad +\frac{\eta_{k}}{1-\lambda} \EE_{{\cal F}^{k-1}}\Pscal{\grd {\cal L}_{\delta_{k}}(\prm_k)}{\grd {\cal L}_{\delta_{k}}(\prm_k)-(1-\lambda)\sum_{m=1}^{\tau_{k}}\lambda^{\tau_{k}-m}g_{k}^{(m)}} 
    \\
    &\quad + \frac{d}{2\delta_{k}}(L_0 + 2 L_1 \G)\eta_{k}^2\EE_{{\cal F}^{k-1}}\norm{\sum_{m=1}^{\tau_{k}}\lambda^{\tau_{k}-m}g_{k}^{(m)}}^2
\end{align*}
By Lemma \ref{lem:UnbiasedGrad}, we have $\EE_{Z\sim\Pi_{\cprm_k},u_{k}} [g_{\delta_{k}}(\prm_k;u_k,Z)]=\grd {\cal L}_{\delta_{k}}(\prm_{k})$,
then by dividing and summing over $k$, it holds that
\begin{align*}
    &(1-\lambda)\sum_{k=0}^{t}\EE\norm{\grd {\cal L}_{\delta_{k}}(\prm_k)}^2 
    \\
    &\leq \sum_{k=0}^{t} \frac{1-\lambda}{\eta_{k}} \EE\left[ {\cal L}_{\delta_{k}}(\prm_{k})- {\cal L}_{\delta_{k+1}}(\prm_{k+1}) + {\cal L}_{\delta_{k+1}}(\prm_{k+1}) - {\cal L}_{\delta_{k}}(\prm_{k+1})  \right]
    \\
    &\quad +(1-\lambda)\sum_{k=0}^{t} \EE\Pscal{\grd {\cal L}_{\delta_{k}}(\prm_k)}{\sum_{m=1}^{\tau_{k}}\lambda^{\tau_{k}-m}\Big(\EE_{Z\sim\Pi_{\cprm_k}} [g_{\delta_{k}}(\prm_k;u_k,Z)]-g_{k}^{(m)}\Big)}
    \\
    &\quad +\sum_{k=0}^{t}\lambda^{\tau_{k}}\EE\norm{\grd {\cal L}_{\delta_{k}}(\prm_k)}^2
    \\
    &\quad + \frac{d (L_0 + 2 L_1 \G)(1-\lambda)}{2}\sum_{k=0}^{t}\frac{\eta_{k}}{\delta_{k}}\EE\norm{\sum_{m=1}^{\tau_{k}}\lambda^{\tau_{k}-m}g_{k}^{(m)}}^2
    \\
    &\eqdef \term{5} + \term{6} + \term{7} + \term{8}
\end{align*}
After splitting RHS into $\term{5},\term{6},\term{7},\term{8}$, we can bound them separately. 

Under Assumption \ref{assu:BoundLoss} and the estimate $\delta_{k}-\delta_{k+1}=\Theta(k^{-\beta-1})$, it holds that
\begin{align*}
\term{5} &=(1-\lambda)\sum_{k=0}^{t}\frac{1}{\eta_{k}}\EE\left[{\cal L}_{\delta_{k}}(\prm_{k})-{\cal L}_{\delta_{k+1}}(\prm_{k+1})\right]+(1-\lambda)\sum_{k=0}^{t}\frac{1}{\eta_{k}}\EE\left[{\cal L}_{\delta_{k+1}}(\prm_{k+1})-{\cal L}_{\delta_{k}}(\prm_{k+1})\right]
\\
&\overset{(a)}{\leq} (1-\lambda ) \G \frac{2}{\eta_{t+1}} + (1-\lambda)\sum_{k=0}^{t}\EE\frac{{\cal L}_{\delta_{k+1}}(\prm_{k}) - {\cal L}_{\delta_{k}}(\prm_{k})}{\eta_{k}}
\\
&\overset{(b)}{\leq} (1-\lambda ) \G \frac{2}{\eta_{t+1}} + (1-\lambda)(L_0+2 L_1\G) \sum_{k=0}^{t}\frac{\delta_{k}-\delta_{k+1}}{\eta_{k}}
\\
&={\cal O}\left((1+t)^{\alpha}+(1+t)^{\alpha-\beta}\right)={\cal O}\left((1+t)^{\alpha}\right)
\end{align*}
where we apply the summation by part in inequality (a) as in Lemma \ref{lem:bound_term_1}, and use the fact $|{\cal L}_{\delta_1}(\prm)-{\cal L}_{\delta_2}(\prm)|\leq\EE_{w}|{\cal L}(\prm+\delta_1 w)-{\cal L}(\prm+\delta_2 w)|\leq (L_0+2 L_1\G)  |\delta_1-\delta_2|$ in inequality (e), as a consequence of Lipschitz continuity. 

As for $\term{6}$, if we let ${\cal B}(t)\eqdef\frac{1}{1+t}\sum_{k=0}^{t}\EE\norm{\grd{\cal L}_{\delta_{k}}(\prm_{k})}^2$, by definition of $g_{k}^{(m)}$, we can split the term as follows
\begin{align*}
&\EE_{{\cal F}^{k-1}} \frac{d}{\delta_{k}}  \left(\EE_{Z\sim \Pi_{\cprm_k}}[\ell(\cprm_k; Z)|\cprm_{k}] - \EE[\ell(\cprm_{k}^{(m)}; Z_{k}^{(m)})|\cprm_{k}^{(m)}, Z_{k}^{(0)}] \right)
    \\
    &= \EE_{{\cal F}^{k-1}}\frac{d}{\delta_{k}} \EE_{Z\sim\Pi_{\cprm_{k}}} \left[ \ell(\cprm_k; Z) - \ell(\cprm_{k}^{(m)}; Z)|\cprm_{k}^{(m)}, \cprm_{k}\right] 
    \\
    &\quad + \EE_{{\cal F}^{k-1}} \frac{d}{\delta_{k}}   \left(\EE_{Z\sim \Pi_{\cprm_k}}[\ell(\cprm_{k}^{(m)}; Z)|\cprm_{k}^{(m)}] - \EE_{\Tilde{Z}_{k}^{(m)}}[\ell(\cprm_{k}^{(m)}; \Tilde{Z}_{k}^{(m)})|\cprm_{k}^{(m)}, \Tilde{Z}_{k}^{(0)}] \right)
    \\
    &\quad + \EE_{{\cal F}^{k-1}}\frac{d}{\delta_{k}}  \left(\EE_{\Tilde{Z}_{k}^{(m)}}[\ell(\cprm_{k}^{(m)}; \Tilde{Z}_{k}^{(m)})|\cprm_{k}^{(m)}, \Tilde{Z}_{k}^{(0)}]  - \EE[\ell(\cprm_{k}^{(m)}; Z_{k}^{(m)})|\cprm_{k}^{(m)}, Z_{k}^{(0)}] \right)
\end{align*}
By applying Jensen's inequality and triangle inequality according to the above splitting, it holds that
\begin{align*}
&\norm{\EE_{{\cal F}^{k-1}}\EE_{Z\sim\Pi_{\cprm_k}} [g_{\delta_{k}}(\prm_k;u_k,Z)]-g_{k}^{(m)}}
\\
=&\left|\frac{d}{\delta_{k}}|\EE_{{\cal F}^{k-1}}   \EE_{Z\sim \Pi_{\cprm_k}}[\ell(\cprm_k; Z)|\cprm_{k}] - \EE[\ell(\cprm_{k}^{(m)}; Z_{k}^{(m)})|\cprm_{k}^{(m)}, Z_{k}^{(0)}] \right|
\\
\leq&\EE_{{\cal F}^{k-1}}\frac{d}{\delta_{k}}\left|   \EE_{Z\sim \Pi_{\cprm_k}}[\ell(\cprm_k; Z)|\cprm_{k}] - \EE[\ell(\cprm_{k}^{(m)}; Z_{k}^{(m)})|\cprm_{k}^{(m)}, Z_{k}^{(0)}] \right|
\\
\leq& \EE_{{\cal F}^{k-1}}  \frac{d}{\delta_{k}}\left|\EE_{Z\sim\Pi_{\cprm_{k}}} \left[ \ell(\cprm_k; Z) - \ell(\cprm_{k}^{(m)}; Z)|\cprm_{k}^{(m)}, \cprm_{k}\right]\right|
    \\
    &\quad + \EE_{{\cal F}^{k-1}} \frac{d}{\delta_{k}}   \left|\EE_{Z\sim \Pi_{\cprm_k}}[\ell(\cprm_{k}^{(m)}; Z)|\cprm_{k}^{(m)}] - \EE_{\Tilde{Z}_{k}^{(m)}}[\ell(\cprm_{k}^{(m)}; \Tilde{Z}_{k}^{(m)})|\cprm_{k}^{(m)}, \Tilde{Z}_{k}^{(0)}] \right|
    \\
    &\quad + \EE_{{\cal F}^{k-1}}\frac{d}{\delta_{k}} \left|\EE_{\Tilde{Z}_{k}^{(m)}}[\ell(\cprm_{k}^{(m)}; \Tilde{Z}_{k}^{(m)})|\cprm_{k}^{(m)}, \Tilde{Z}_{k}^{(0)}]  - \EE[\ell(\cprm_{k}^{(m)}; Z_{k}^{(m)})|\cprm_{k}^{(m)}, Z_{k}^{(0)}] \right|
    \\
    \overset{(c)}{\leq} &\frac{d}{\delta_{k}}\EE_{{\cal F}^{k-1}}L_0\norm{\cprm_{k}^{(m)}-\cprm_{k}}
    \\
    & \quad + \frac{2d\G}{\delta_{k}}\EE_{{\cal F}^{k-1}} \tv{\Pi_{\prm_{k}}}{\PP(\hat{Z}_{k}^{(m)}\in\cdot|\cprm_{k}^{(0)},\hat{Z}_{k}^{(0)})}\\
    &\quad + \frac{2d\G}{\delta_{k}}\EE_{{\cal F}^{k-1}} \tv{\PP(\hat{Z}_{k}^{(m)}\in\cdot|\cprm_{k}^{(0)},\hat{Z}_{k}^{(0)})}{\PP(Z_{k}^{(m)}\in\cdot|\cprm_{k}^{(0)},Z_{k}^{(0)})}
    \\
    \overset{(d)}{\leq} & \frac{d L_0}{\delta_{k}} \EE_{{\cal F}^{k-1}}\norm{\cprm_{k}^{(m)}-\cprm_{k}}+\frac{2d\G}{\delta_{k}}M\rho^{m}+\frac{2d L_2\G}{\delta_{k}}\EE_{{\cal F}^{k-1}}\sum_{\ell=1}^{m-1}\norm{\cprm_{k}^{(\ell)}-\cprm_{k}}
    \\
    \leq & \frac{d L_0}{\delta_{k}} d\G\sum_{j=1}^{m-1}\lambda^{\tau_{k}-j}\frac{\eta_{k}}{\delta_{k}}+\frac{2d\G M}{\delta_{k}}\rho^{m}+\frac{2d L_{2}\G}{\delta_{k}}d\G\sum_{\ell=1}^{m-1}\sum_{j=1}^{\ell-1}\lambda^{\tau_{k}-j}\frac{\eta_{k}}{\delta_{k}}
    \\
    < & d^2 L_0\G\frac{\eta_{k}}{\delta_{k}^2}\frac{\lambda^{\tau_{k}-m+1}}{1-\lambda}+\frac{2d\G M}{\delta_{k}}\rho^{m}+2d^2 L_2\G^2\frac{\eta_{k}}{\delta_{k}^2}\frac{\lambda^{\tau_{k}-m+2}}{(1-\lambda)^2}
\end{align*}
where inequality (c) is due to Lipschitzness of decoupled risk, inequality (d) is due to Assumption \ref{assu:FastMixing} and  Lemma \ref{lem:tv_summation_bound} (a consequence of Assumption \ref{assu:smooth_kernel}). 
Given $\tau_{k}\geq\frac{\log(1+k)}{\log 1/\max\{\rho,\lambda\}}$, then the following deterministic bound holds for all $k>0$,
\begin{align*}
    & \EE_{{\cal F}^{k-1}}\sum_{m=1}^{\tau_{k}}\lambda^{\tau_{k}-m}\norm{\EE_{Z\sim\Pi_{\cprm_k}} [g_{\delta_{k}}(\prm_k;u_k,Z)]-g_{k}^{(m)}} 
    \\
    \leq & d^2 L_0\G\frac{\eta_{k}}{\delta_{k}^2}\frac{\lambda}{1-\lambda}\sum_{m=1}^{\tau_{k}}\lambda^{\tau_{k}-m}+2d^2 L_2\G^2\frac{\eta_{k}}{\delta_{k}^2}\frac{\lambda^2}{1-\lambda}\sum_{m=1}^{\tau_{k}}\lambda^{\tau_{k}-m} + 2d\G M/\delta_{k}\sum_{m=1}^{\tau_{k}}\rho^m\lambda^{\tau_{k}-m}
    \\
    < & d^2 \frac{\lambda}{(1-\lambda)^2} L_0\G\frac{\eta_{k}}{\delta_{k}^2}+2d^2 \frac{\lambda^2}{(1-\lambda)^2} L_2\G^2\frac{\eta_{k}}{\delta_{k}^2}+2d \G M/\delta_{k}\sum_{m=1}^{\tau_{k}}\max\{\rho,\lambda\}^{\tau_{k}}
    \\
    \leq & d^2 \frac{\lambda}{(1-\lambda)^2} L_0\G\frac{\eta_{k}}{\delta_{k}^2}+2d^2 \frac{\lambda^2}{(1-\lambda)^2} L_2\G^2\frac{\eta_{k}}{\delta_{k}^2}+2d \G M \frac{\tau_{k}}{(1+k)\delta_{k}}
\end{align*}
So for sufficiently large $t$, it holds that
\begin{align*}
    \term{6}&\leq (1-\lambda)\sum_{k=0}^{t}\EE\norm{\grd{\cal L}(\prm_{k})}\EE_{{\cal F}^{k-1}}\norm{\sum_{m=1}^{\tau_{k}}\lambda^{\tau_{k}-m}\EE_{Z\sim\Pi_{\cprm_k}} [g_{\delta_{k}}(\prm_k;u_k,Z)]-g_{k}^{(m)}} 
    \\
    &\leq (1-\lambda)\sum_{k=0}^{t}\EE\norm{\grd{\cal L}(\prm_{k})}\sum_{m=1}^{\tau_{k}}\lambda^{\tau_{k}-m}\EE_{{\cal F}^{k-1}}\norm{\EE_{Z\sim\Pi_{\cprm_k}} [g_{\delta_{k}}(\prm_k;u_k,Z)]-g_{k}^{(m)}} 
    \\
    &\leq \sum_{k=0}^{t}\EE\norm{\grd{\cal L}(\prm_{k})} d^2 \frac{\lambda}{1-\lambda}((L_0+2 L_1\G) \G+2 L_1\G+2\lambda L_2\G^2)\frac{\eta_{k}}{\delta_{k}^2}
    +\sum_{k=0}^{t}\EE\norm{\grd{\cal L}(\prm_{k})}2d \G M \frac{\tau_{k}}{(1+k)\delta_{k}}
    \\
    &\leq\sum_{k=0}^{t}\EE\norm{\grd{\cal L}(\prm_{k})} d^2 \frac{2\lambda}{(1-\lambda)^2}(L_0 \G+2 L_1\G+2\lambda L_2\G^2)\frac{\eta_{k}}{\delta_{k}^2}
    \\
    &= d^2 \frac{2\lambda}{(1-\lambda)^2}( L_0\G+2 L_1\G+2\lambda L_2\G^2) \sum_{k=0}^{t}\EE\norm{\grd{\cal L}(\prm_{k})}\frac{\eta_{k}}{\delta_{k}^2}
    \\
    &\leq d^2 \frac{2\lambda}{(1-\lambda)^2}(L_0 \G+2 L_1\G+2\lambda L_2\G^2) \Big(\sum_{k=0}^{t}\EE\norm{\grd{\cal L}(\prm_{k})}^2\Big)^{1/2} \Big(\sum_{k=0}^{t}\frac{\eta_{k}^2}{\delta_{k}^4}\Big)^{1/2}
    \\
    &\leq c_9 d^2 {\cal B}(t)^{1/2} (1+t)^{\frac{1}{2}+\frac{1}{2}-(\alpha-2\beta)}
\end{align*}
Therefore, there exists a constant $c_9>0$ such that
\[
\term{6} \leq c_9 d^2 {{\cal B}(t)}^{1/2}(1+t)^{1-(\alpha-2\beta)}
\]
where there is an extra $\beta/2$ in exponent because the $L$ in $c_2$ is now a variable $d (L_0+2 L_1\G) /\delta_{k}$.

For ($L_0+2 L_1\G$)-Lipschitz continuous ${\cal L}(\prm)$, for all $\delta>0$ it holds that $\norm{\grd {\cal L}_{\delta}(\prm)}\leq (L_0+2 L_1\G) $. Given $\tau_{k}\geq\frac{\log(1+k)}{\log 1/\max\{\rho,\lambda\}}$, it holds that $\lambda^{\tau_{k}}\EE\norm{\grd {\cal L}_{\delta_{k}}(\prm_{k})}^2\leq \frac{d L^2}{\delta_{0}(1+k)}$,  then $\term{7}$ can be bounded as follows
\[
\term{7} \leq \frac{d L^2}{\delta_{0}}\sum_{k=0}^{t} (1+k)^{-1}={\cal O}(\log(1+t))
\]

$\term{8}$ is similar to $\term{4}$. For all $0\leq k\leq t, 1\leq m\leq \tau_{k}$, it holds that $\norm{g_{k}^{(m)}}\leq\frac{d \G}{\delta_{k}}$, which implies
\begin{align*}
\term{8}&\leq (1-\lambda)\frac{d (L_0 + 2 L_1 \G)}{2}\sum_{k=0}^{t}\frac{\eta_{k}}{\delta_{k}}\EE\norm{\sum_{m=1}^{\tau_{k}}\lambda^{\tau_{k}-m}g_{k}^{(m)}}^2
\\
&\leq (1-\lambda)\frac{d (L_0 + 2 L_1 \G)}{2}\sum_{k=0}^{t}\frac{\eta_{k}}{\delta_{k}}\EE\Big(\sum_{m=1}^{\tau_{k}}\lambda^{\tau_{k}-m}\norm{g_{k}^{(m)}}\Big)^2
\\
&\leq (1-\lambda)\frac{d (L_0 + 2 L_1 \G)}{2}\sum_{k=0}^{t}\frac{\eta_{k}}{\delta_{k}}\Big(\sum_{m=1}^{\tau_{k}}\lambda^{\tau_{k}-m}\frac{d \G}{\delta_{k}}\Big)^2
\\
&= (1-\lambda)\frac{d^3 (L_0 + 2 L_1 \G) \G^2}{2}\sum_{k=0}^{t}\frac{\eta_{k}}{\delta_{k}^3}\Big(\sum_{m=1}^{\tau_{k}}\lambda^{\tau_{k}-m}\Big)^2
\\
&\leq \frac{d^3 (L_0 + 2 L_1 \G) \G^2}{2(1-\lambda)}\sum_{k=0}^{t}\frac{\eta_{k}}{\delta_{k}^3}\leq c_{10} (1+t)^{1-\left(\alpha-3\beta\right)}
\end{align*}
where $c_{10}>0$ is a constant hiding the factor $\frac{\eta_{0}}{\delta_{0}^3}$.

Applying quadratic technique in Lemma \ref{lem:major_bound}, and for all $\alpha,\beta$ satisfying $0<3\beta<\alpha<1$, it is clear that only $\term{5}$ and $\term{8}$ contribute to the asymptotic rate, so for all $t\geq t_4$ (for some constant $t_4>0$), we have
\[
\frac{1}{1+T}\sum_{k=0}^{T}\EE\norm{\grd{{\cal L}_{\delta_{k}}(\prm_k)}}^2 
    ={\cal O}(T^{-\min\{1-\alpha, \alpha-3\beta\}})
\]
The error estimate directly follows from Lemma \ref{lem:composite_lip}. 
\end{proof}
\begin{remark}
Note that Corrollary \ref{cor:non-smooth} follows directly from Theorem \ref{thm:non-smooth} by setting $\alpha=3/4$ and $\beta=1/6$.
\end{remark}
\section{Auxiliary Lemmas}

\begin{lemma}\label{lem:UnbiasedGrad} {\bf (Smoothing)} For continuous ${\cal L}(\prm):\mathbb{R}^d\rightarrow\mathbb{R}$, its smoothed approximation ${\cal L}_{\delta}(\prm)\eqdef\EE_{w\sim \unif(\BB)}[{\cal L}(\prm+\delta w)]$ is differentiable, and it holds that 
\[
\EE_{\substack{u\sim \unif(\SS),\\ Z\sim \Pi_{\prm+\delta u}}} [g_{\delta}(\prm; u, Z)]=\grd {\cal L}_{\delta}(\prm)
\]
Moreover, if ${\cal L}(\prm)$ is $\bar{L}$-Lipschitz continuous, then ${\cal L}_{\delta}(\prm)$ is $\frac{d}{\delta} \bar{L}$-smooth.
\end{lemma}
\begin{proof}
The first fact follows from (generalized) Stoke's theorem. Given continuous ${\cal L}(\prm)$, it holds that
\beq\label{eqn:stokes}
    \grd\int_{\delta\BB}{\cal L}(\prm + v)\diff v = \int_{\delta \SS}{\cal L}(\prm + r) \frac{r}{\norm{r}}\diff r
\eeq
Observe that the RHS of Equation (\ref{eqn:stokes}) is continuous in $\prm$, which implies ${\cal L}_{\delta}(\prm)=\frac{1}{{\sf vol}(\delta\BB)}\int_{\delta\BB}{\cal L}(\prm + v)\diff v $ is differentiable. 
Note that the volume to surface area ratio of $\delta\BB$ is $\delta/d$, so it follows from Equation (\ref{eqn:stokes}) that 
\begin{align*}
\grd {\cal L}_{\delta}(\prm)&=\frac{{\sf vol}(\delta\SS)}{{\sf vol}(\delta\BB)}\int_{\delta \SS}{\cal L}(\prm + r) \frac{r}{{\sf vol}(\delta\SS)\norm{r}}\diff r
=\frac{d}{\delta}\EE_{u\sim \unif(\SS)}[{\cal L}(\prm+\delta u)u]
\\
&=\EE_{u\sim \unif(\SS)}\EE_{Z\sim\pi_{\prm+\delta u}}[\frac{d}{\delta}\ell(\prm+\delta u;Z)u]
=\EE_{\substack{u\sim \unif(\SS),\\ Z\sim \Pi_{\prm+\delta u}}} [g_{\delta}(\prm; u, Z)]
\end{align*}
where we use the definition of $g_{\delta}(\prm;u,z)$ in the last equality.

If further assuming ${\cal L}(\prm)$ is $\bar{L}$-Lipschitz continuous, then we obtain
\begin{align*}
    \norm{\grd {\cal L}_{\delta}(\prm_1)-\grd {\cal L}_{\delta}(\prm_2)}
    &=\frac{d}{\delta}\cdot \norm{\frac{1}{{\sf vol}(\SS)}\int_{\SS} \left[{\cal L}(\prm_1+\delta u)-{\cal L}(\prm_2+\delta u)\right] u \diff u}
    \\
    &\leq \frac{d}{\delta}\cdot  \bar{L} \norm{\prm_1-\prm_2}.
\end{align*}

\end{proof}

\begin{lemma}\label{lem:obs}{\bf (${\cal O}(\delta)$-Biased Gradient Estimation)}\label{lem:BoundedBias}

Under Assumptions \ref{assu:Lip},  fix a proximity parameter $\delta>0$, it holds that
\begin{align*}
    \norm{\EE_{\substack{u\sim \unif(\SS),\\ Z\sim \Pi_{\prm+\delta u}}} [g_{\delta}(\prm; u, Z)] - \grd {\cal L}(\prm) } &= \norm{\grd {\cal L}_{\delta}(\prm) - \grd {\cal L}(\prm)}
        \leq \delta L
\end{align*}
\end{lemma}

\begin{proof}
By Lemma \ref{lem:UnbiasedGrad}, we have
\[
\EE_{\substack{u\sim \unif(\SS),\\ Z\sim\Pi_{\prm+\delta u}}} [g(\prm;u,Z)] = \grd {\cal L}_{\delta}(\prm)
\]

Note that when ${\cal L}(\prm)$ is differentiable, we have
\[
\grd {\cal L}_{\delta}(\prm)= \grd \left[\EE_{w\sim\unif(\BB)}{\cal L}(\prm+\delta w)\right]=\EE_{w\sim\unif(\BB)}\grd{\cal L}(\prm+\delta w)
\]
Then under Assumption \ref{assu:Lip}, by linearity of expectation and Jensen's inequality, it holds that
\begin{align*}
    \norm{\grd {\cal L}_{\delta}(\prm)-\grd {\cal L}(\prm)}=\norm{\EE_{w\sim \unif(\BB)}[\grd {\cal L}(\prm +\delta w) - \grd {\cal L}(\prm)]}
        \leq \delta L.
\end{align*}
\end{proof}

\begin{Corollary}\label{cor:bdd_grad}
Under Assumption \ref{assu:Lip} and \ref{assu:BoundLoss}, for all $\prm\in\mathbb{R}^d$, it holds that
\[
    \norm{\grd{\cal L}(\prm)}\leq 2\sqrt{LG}
\]
\end{Corollary}
\begin{proof}
Omitted.
\end{proof}

\begin{lemma}{\bf (Lipschitz Continuity of Decoupled Risk)}\label{lem:lip_decoupled_risk}
Under Assumption \ref{assu:Lip}, \ref{assu:BoundLoss} and \ref{assu:smooth_dist}, it holds that
\[
\left|\EE_{Z\sim \Pi_{\prm_2}}\left[\ell(\prm_1;Z)-\ell(\prm_2;Z)\right]\right|\leq 2(\G L_1+\sqrt{L\G})\norm{\prm_1-\prm_2}+\frac{L}{2}\norm{\prm_1-\prm_2}^2
\]
\end{lemma}
\begin{proof}
Let ${\cal L}(\prm_1,\prm_2)\eqdef\EE_{Z\sim\Pi_{\prm_2}}\ell(\prm_1;Z)$ denote the decoupled performative risk, then we have
\begin{align*}
\text{LHS}
&=\left|{\cal L}(\prm_1,\prm_2)-{\cal L}(\prm_2, \prm_2)\right|
\\
&\leq \left|{\cal L}(\prm_1)-{\cal L}(\prm_2)\right| + \left|{\cal L}(\prm_1,\prm_2)-{\cal L}(\prm_1,\prm_1)\right|
\\
&\leq \left|{\cal L}(\prm_1)-{\cal L}(\prm_2)-\Pscal{\grd{\cal L}(\prm_2)}{\prm_1-\prm_2}\right|+\left|\Pscal{\grd{\cal L}(\prm_2)}{\prm_1-\prm_2}\right|
+\left|{\cal L}(\prm_1,\prm_1)-{\cal L}(\prm_1, \prm_2)\right|
\\
&\overset{(a)}{\leq}\frac{L}{2}\norm{\prm_1-\prm_2}^2+\left|\Pscal{\grd{\cal L}(\prm_2)}{\prm_1-\prm_2}\right|
+\left|{\cal L}(\prm_1,\prm_1)-{\cal L}(\prm_1, \prm_2)\right|
\\
&\overset{(b)}{\leq}\frac{L}{2}\norm{\prm_1-\prm_2}^2+2\sqrt{L\G}\norm{\prm_1-\prm_2}
+\left|{\cal L}(\prm_1,\prm_1)-{\cal L}(\prm_1, \prm_2)\right|
\\
&=\frac{L}{2}\norm{\prm_1-\prm_2}^2+2\sqrt{ L\G}\norm{\prm_1-\prm_2}
+\left|\int\ell(\prm_1;z)\left(\Pi_{\prm_1}(z)-\Pi_{\prm_2}(z)\right)dz\right|
\\
&\overset{(c)}{\leq}\frac{L}{2}\norm{\prm_1-\prm_2}^2+2\sqrt{L\G}\norm{\prm_1-\prm_2}
+2\G\tv{\Pi_{\prm_1}}{\Pi_{\prm_2}}
\\
&\leq\frac{L}{2}\norm{\prm_1-\prm_2}^2+2\sqrt{L\G}\norm{\prm_1-\prm_2}
+2\G L_1\norm{\prm_1-\prm_2}
\\
&=2\left(\sqrt{L\G}
+\G L_1\right)\norm{\prm_1-\prm_2}+\frac{L}{2}\norm{\prm_1-\prm_2}^2
\end{align*}
where we use Assumption \ref{assu:Lip} in inequality (a), Corollary \ref{cor:bdd_grad} in inequality (b), Assumption \ref{assu:BoundLoss} in inequality (c), and Assumption \ref{assu:smooth_dist} in the last inequality.
\end{proof}

\begin{lemma}\label{lem:tv_summation_bound}
    Under Assumption \ref{assu:smooth_kernel}, it holds that for all $0\leq\ell\leq m$, $m\geq 1$
    \[
    \tv{\PP(Z_{k}^{(\ell+1)}\in \cdot|Z_{k}^{(0)})}{\PP(\tilde{Z}_{k}^{(\ell +1)}\in \cdot|Z_{k}^{(0)})}
    \leq 
    L_2 \norm{\cprm_{k}^{(\ell)} -\cprm_k} + \tv{\PP(Z_{k}^{(\ell)} \in \cdot|Z_{k}^{(0)})}{\PP(\Tilde{Z}_{k}^{(\ell)}\in \cdot|Z_{k}^{(0)})}
    \]
Unfold above recursion leads to the following inequality, 
\[
    \tv{\PP(Z_{k}^{(m)}\in \cdot|Z_{k}^{(0)})}{\PP(\tilde{Z}_{k}^{(m)}\in \cdot|Z_{k}^{(0)})}
    \leq L_2 \sum_{\ell=1}^{m-1} \norm{\cprm_{k}^{(\ell)} -\cprm_{k} }, \quad \forall m\geq 1.
\]
\end{lemma}

\begin{proof}
Recall the notation $\cprm_{k}^{(\ell)}=\cprm_{k}^{(\ell)}+\delta_{k}u_{k}, \cprm_{k} = \cprm_{k}+\delta_{k}u_{k}$, and the fact that $Z_{k}=Z_{k}^{(0)}=\Tilde{Z}_{k}^{(0)}$, we have
\begin{align*}
    2\cdot\text{LHS} &= \int_{{\sf Z}} \left| \PP(Z_{k}^{(\ell+1)}=z|Z_{k}^{(0)}) - \PP(\Tilde{Z}_{k}^{(\ell+1)}=z|Z_{k}^{(0)})\right| \diff z
    \\
    &= \int_{{\sf Z}} \left| \int_{{\sf Z}}
        \PP({Z}_{k}^{(\ell)}=z^\prime, Z_{k}^{(\ell+1)} = z|Z_{k}^{(0)}) - \PP(\Tilde{Z}_{k}^{(\ell)}=z^\prime, \Tilde{Z}_{k}^{(\ell+1)} = z|Z_{k}^{(0)})
    \diff z^\prime \right| \diff z
    \\
    &\leq \int_{\sf Z}\int_{\sf Z}  \left|
        \TT_{\cprm_k^{(\ell)}}(z^\prime, z) \PP(Z_{k}^{(\ell)} = z^\prime|Z_{k}^{(0)}) -  \TT_{\cprm_{k}}(z^\prime, z) \PP(\Tilde{Z}_{k}^{(\ell)} = z^\prime|Z_{k}^{(0)})
    \right| \diff z^\prime \diff z
    \\
    &\leq \int_{\sf Z}\int_{\sf Z} \left|
        \TT_{\cprm_{k}^{(l\ell}}(z^\prime, z) \PP(Z_{k}^{(\ell)} = z^\prime|Z_{k}^{(0)}) -  \TT_{\cprm_{k}}(z^\prime, z) \PP(Z_{k}^{(\ell)} = z^\prime|Z_{k}^{(0)})
    \right| \diff z^\prime \diff z
    \\
        &\quad +
        \int_{\sf Z}\int_{\sf Z}  \left|
            \TT_{\cprm_k}(z^\prime, z) \PP(Z_{k}^{(\ell)} = z^\prime|Z_{k}^{(0)}) -  \TT_{\cprm_{k}}(z^\prime, z) \PP(\Tilde{Z}_{k}^{(\ell)} = z^\prime|Z_{k}^{(0)})
        \right| \diff z^\prime \diff z
    \\
    &\overset{(a)}{=} 
    \int_{\sf Z} \PP(Z_{k}^{\ell}=z^\prime|Z_{k}^{(0)}) \int_{\sf Z} \left| \TT_{\cprm_k}(z^\prime, z) - \TT_{\cprm_{k}^{(\ell)}}(z^\prime, z)\right| \diff z\diff z^\prime
    \\
    &\quad + 
    \int_{\sf Z}\left[\int_{\sf Z} \TT_{\cprm_k}(z^\prime, z)\diff z\right] 
    \left| \PP(Z_{k}^{(\ell)}=z^\prime|Z_{k}^{(0)}) - \PP(\Tilde{Z}_{k}^{(\ell)} = z^\prime|Z_{k}^{(0)})\right| \diff z^\prime
    \\
    &\leq  \int_{\sf Z} \PP(Z_{k}=z^\prime|Z_{k}^{(0)}) \cdot 2 \tv{\TT_{\cprm_{k}}(z^\prime, \cdot)}{\TT_{\cprm_{k}}(z^\prime, \cdot)} \diff z^\prime + 2\tv{\PP(Z_{k}^{(\ell)} \in \cdot|Z_{k}^{(0)})}{\PP(\Tilde{Z}_{k}^{(\ell)}\in \cdot|Z_{k}^{(0)})}
    \\
    &\leq 2\int_{\sf Z} \PP(Z_{k}^{(\ell)} =z^\prime|Z_{k}^{(0)}) \diff z^\prime \cdot L_2 \norm{\cprm_{k}^{(\ell)} - \cprm_{k}} 
    + 2\tv{\PP(Z_{k}^{(\ell)} \in \cdot|Z_{k}^{(0)})}{\PP(\Tilde{Z}_{k}^{(\ell)} \in \cdot|Z_{k}^{(0)})}
    \\
    &= 2\left[
        L_2  \norm{\cprm_{k}^{(\ell)} -\cprm_k} + \tv{\PP(Z_{k}^{(\ell)} \in \cdot|Z_{k}^{(0)})}{\PP(\Tilde{Z}_{k}^{(\ell)}\in \cdot|Z_{k}^{(0)})}
    \right]=2\cdot\text{RHS}
\end{align*}
where inequality (a) holds due to the (absolutely) integrable condition (which automatically holds for probability density functions and kernels), and Assumption \ref{assu:smooth_kernel} is used in the last inequality.
\end{proof}

\begin{assumption} \label{assu:alternative}
Assume that there exists constants $L_0, L_1$ such that:
    \begin{enumerate}[label=(\roman*)]
    \item \label{assu:replace1} $|\ell(\prm;z)-\ell(\prm;z^\prime)|\leq L_0\norm{z-z^{\prime}}$ for any $ \prm \in \RR^d$, $z,z^{\prime} \in {\sf Z}$,
    \item \label{assu:replace2}  $W_1(\Pi_{\prm},\Pi_{\prm^{\prime}})\leq L_1\norm{\prm-\prm^{\prime}}$ for any $\prm , \prm^{\prime} \in \RR^d$, where $W_1(\Pi,\Pi')$ denotes the Wasserstein-1 distance between the distributions $\Pi, \Pi'$.
\end{enumerate}
\end{assumption}

We observe that a similar result to Lemma~\ref{lem:lip_decoupled_risk} can be proven by replacing Assumption~\ref{assu:smooth_dist} with Assumption~\ref{assu:alternative}:
\begin{lemma}{\bf (Lipschitz Continuity of Decoupled Risk, Alternative Condition based on Wasserstein-1 distance.)}\label{lem:lip_decoupled_risk_wasserstein}
Under Assumption \ref{assu:Lip}, \ref{assu:BoundLoss}, \ref{assu:alternative}.
Then, for any $\prm_1, \prm_2 \in \RR^d$, it holds that
\[
\left|\EE_{Z\sim \Pi_{\prm_2}}\left[\ell(\prm_1;Z)-\ell(\prm_2;Z)\right]\right|\leq (L_0 L_1+2\sqrt{L\G})\norm{\prm_1-\prm_2}+\frac{L}{2}\norm{\prm_1-\prm_2}^2.
\]
\end{lemma}
\begin{proof}
Observe that
    \begin{align*}
    \text{LHS}
&=\left|{\cal L}(\prm_1,\prm_2)-{\cal L}(\prm_2, \prm_2)\right|
\\
&\leq \left|{\cal L}(\prm_1)-{\cal L}(\prm_2)\right| + \left|{\cal L}(\prm_1,\prm_2)-{\cal L}(\prm_1,\prm_1)\right|
\\
&\leq \left|{\cal L}(\prm_1)-{\cal L}(\prm_2)-\Pscal{\grd{\cal L}(\prm_2)}{\prm_1-\prm_2}\right|+\left|\Pscal{\grd{\cal L}(\prm_2)}{\prm_1-\prm_2}\right|
+\left|{\cal L}(\prm_1,\prm_1)-{\cal L}(\prm_1, \prm_2)\right|
\\
&\leq\frac{L}{2}\norm{\prm_1-\prm_2}^2+\left|\Pscal{\grd{\cal L}(\prm_2)}{\prm_1-\prm_2}\right|
+\left|{\cal L}(\prm_1,\prm_1)-{\cal L}(\prm_1, \prm_2)\right|
\\
&\leq \frac{L}{2}\norm{\prm_1-\prm_2}^2+2\sqrt{L\G}\norm{\prm_1-\prm_2}
+\left|{\cal L}(\prm_1,\prm_1)-{\cal L}(\prm_1, \prm_2)\right|
\\
&=\frac{L}{2}\norm{\prm_1-\prm_2}^2+2\sqrt{ L\G}\norm{\prm_1-\prm_2}
+\left|\EE_{Z\sim\Pi_{\prm_1},Z^{\prime}\sim\Pi_{\prm_2}}[\ell(\prm_1;Z)-\ell(\prm_1;Z^{\prime})]\right|
\\
&\overset{(a)}{\leq}\frac{L}{2}\norm{\prm_1-\prm_2}^2+2\sqrt{L\G}\norm{\prm_1-\prm_2}
+L_0 W_1(\Pi_{\prm_1},\Pi_{\prm_2})
\\
&\overset{(b)}{\leq}\frac{L}{2}\norm{\prm_1-\prm_2}^2+2\sqrt{L\G}\norm{\prm_1-\prm_2}
+L_0 L_1\norm{\prm_1-\prm_2}
\\
&=\left(L_0 L_1+2\sqrt{L\G}\right)\norm{\prm_1-\prm_2}+\frac{L}{2}\norm{\prm_1-\prm_2}^2
    \end{align*}
    where the inequality (a) is due to \cite[Lemma D.4]{perdomo2020performative} and the alternative assumption \ref{assu:replace1}, the inequality (b) is due to the alternative assumption \ref{assu:replace2} (i.e., Lipschitz condition on distribution map $\Pi_{\prm}$ in Wasserstein-1 metric).
\end{proof}
Consequently, the conclusions in Theorem~\ref{thm1} hold (with slightly different constants) when Assumption~\ref{assu:smooth_dist} is replaced by Assumption~\ref{assu:alternative}. We observe that the former assumption is only used in ensuring the bound in Lemma~\ref{lem:lip_decoupled_risk}; cf.~the proof of Lemma~\ref{lem:bound_term_2}.

\end{document}